\documentclass[a4paper,twoside,12pt]{article}
\usepackage[utf8x]{inputenc}		%Pour de l'UTF8
\usepackage[english]{babel}

\usepackage{tikz,tkz-tab}
	\usetikzlibrary{shapes}
	
\usepackage{amscd}
\usepackage{amsmath,amsthm,amssymb,amsfonts,pdfpages}
\usepackage{version}

\usepackage[colorlinks=true,citecolor=blue,linkcolor=blue]{hyperref}
\usepackage{stmaryrd}				%Crochets pour les intervalles entiers
\usepackage{bbm}					%Fonction indicatrice
\usepackage{version}
\usepackage{geometry}
\usepackage{multirow}
\usepackage{array}
\usepackage{authblk}
	
\input{commath}

\newcommand{\egloi}{\overset{(d)}{=}}

\DeclareMathOperator*{\esssup}{ess\,sup}

\theoremstyle{plain}
\newtheorem{thm}{Theorem}
\newtheorem{cor}[thm]{Corollary}
\newtheorem{lem}[thm]{Lemma}
\newtheorem{prop}[thm]{Proposition}

\title{Deviation results for Mandelbrot's multiplicative cascades with exponential tails}

\author[1,2]{Thierry Klein}
\author[1,3]{Agn\`es Lagnoux}
\author[1,4]{Pierre Petit}

\affil[1]{Institut de Math\'ematiques de Toulouse; UMR5219. Universit\'e de Toulouse; CNRS.}
\affil[2]{ENAC - Ecole Nationale de l'Aviation Civile, Universit\'e de Toulouse, France. E-mail: thierry.klein@math.univ-toulouse.fr}
\affil[3]{UT2J, F-31058 Toulouse, France. E-mail: lagnoux@math.univ-tlse2.fr (corresponding author)}
\affil[4]{UT3, F-31062 Toulouse, France. E-mail: pierre.petit@math.univ-toulouse.fr}

\begin{document}

\maketitle

\begin{abstract}
\noindent Let $W$ be a nonnegative random variable with expectation $1$. For all $r \geqslant 2$, we consider the total mass $Z_r^\infty$ of the associated Mandelbrot multiplicative cascade in the $r$-ary tree. For all $n \geqslant 1$, we also consider the total mass $Z_r^n$ of the measure at height $n$ in the $r$-ary tree. Liu, Rio, Rouault \cite{lrr,liu2000limit,Rouault04} established large deviation results for $(Z_r^n)_{r \geqslant 2}$ for all $n \in \intervallentfo{1}{\infty}$ (resp.\ for $n = \infty$) in case $W$ has an everywhere finite cumulant generating function $\Lambda_W$ (resp.\ $W$ is bounded). Here, we extend these results to the case where $\Lambda_W$ is only finite on a neighborhood of zero. And we establish all deviation results (moderate, large, and very large deviations). It is noticeable that we obtain nonconvex rate functions. Moreover, our proof of upper bounds of deviations for $(Z_r^\infty)_{r \geqslant 2}$ rely on the moment bound instead of the standard Chernoff bound.\\

\noindent \textbf{Keywords:} Mandelbrot multiplicative cascades ; large deviation principles ; heavy tails ; moment bound ; combinatorial problems.\\

\noindent \textbf{AMS MSC 2010:} 60F10; 60F05; 62E20; 60C05; 68W40.
\end{abstract}

\tableofcontents

% \alert{
% \begin{itemize}
% \item RAJOUTER $I^\infty(a) \sim c e \log(a)$. Pour la majoration, on a besoin de $I^1(w) = \Lambda_W^*(w) \leqslant c w$, puis, par récurrence, $I^n(a) \leqslant c n a^{1/n}$. La minoration est obtenue par Markov.
% \\
% By item (ii) of Proposition \ref{prop:upper_bound_large_moment}, for $\eta \in \intervalleoo{0}{e}$,
% \begin{align*}
% \limsup \frac{1}{r}\log \Prob(Z_r\geqslant a) \leqslant \limsup \frac 1r \log \Espe[Z_r^{\eta r}]-\eta \log a\leqslant  c(\eta)-\eta \log a. 
% %\sim -\eta \log a
% \end{align*}
% %as $\eta \to 0$. 
% In particular, 
% \[
% \liminf_{a \to \infty} \frac{\overline{I}^{\infty}(a)}{\log(a)} \geqslant e.
% \]
% %
% \item RÉORGANISER la section 3.1 (propriétés de $I^n$).
% \item voir si on fait une Proposition 17 générique avec $\varepsilon$ et une fonction générique. Idem pour continuité de $I^n$ et $\overline{I}^\infty$
% \item vérifier parties entières si besoin.
% \item Notation $\Lambda_W$ ou $\Lambda$ ?
% \item Contraction principle : unilateral ou non ?
% \item Reprendre tous les énoncés pour $c = \infty$. Notamment Propositions \ref{prop:moments_W}.
% \item VÉRIFIER qu'on n'utilise pas :
% \begin{itemize}
% \item $\Lambda_W^*(w) \leqslant c w$ (voir la preuve de Proposition \ref{prop:Lambda}, item 2.) ;
% \item $\Lambda_W^*(w) \leqslant - \log\Prob(W \geqslant w)$ (par Markov exponentiel cf. proof of Proposition \ref{prop:Lambda}, item 1).
% \end{itemize}
% \end{itemize}
% }

\section{Introduction}

Multiplicative cascades were introduced by Mandelbrot in \cite{man74a,man74b} in order to analyze some problems of turbulence precisely. For $r \in \N^*\setminus \{1\}$, a multiplicative cascade is a random measure $\mu_r^\infty$ on the unit interval, defined as a limit of measures $\mu_r^n$ on the sub-$\sigma$-algebra generated by the $r$-adic intervals of level $n \in \N^*$. It is natural to study $Z_r^n \defeq \mu_r^n(\intervalleff{0}{1})$ the total mass of the measure $\mu_r^n$. It can be described as follows. Let $W$ be a nonnegative random variable such that $\Espe[W]=1$. Now, let $(W_{i_1,\dots, i_n})_{n \geqslant 1, 1 \leqslant i_1, \dots, i_n\leqslant r}$ be a family of independent and identically distributed random variables distributed as $W$ indexed by all finite sequences of integers between $1$ and $r$. 
Then we define, for all $n \in \N^*$,
\begin{align}
Z_r^n
 & \defeq \frac{1}{r^n} \sum_{ 1\leqslant i_1, \dots, i_n\leqslant r} 
W_{i_1} W_{i_1,i_2} \cdots W_{i_1,\dots, i_n} \label{eq:defZrnintro}  %\\
% & = \frac{1}{r} \sum_{i=1}^r W_{i} Z_{r,i}^{n-1} \label{eq:Efiniintro}
\end{align}
and we set $Z_r^0 \defeq 1$.

For fixed $r$, the properties of $Z_r^n$ were studied in several works. First, Kahane and Peyri\`ere \cite{kahane1976} showed that $(Z_r^n)_{n \geqslant 1}$ is a nonnegative martingale with expectation $1$ and  that  the properties of the limit $Z_r^{\infty}$ were characterized by the behavior of the moments of $W$ and the quantity $\Espe[W\log(W)]$. In particular, $\Espe[Z_r^\infty] \leqslant 1$ and it is proved in \cite[Théorème 1]{kahane1976} that $\Espe[Z_r^\infty] = 1$ if and only if $\Espe[W \log(W)] < \log(r)$. Moreover the distribution of $Z_r^{\infty}$ is solution of the distributional equation 
\begin{align}
Z \egloi \frac{1}{r} \sum_{i=1}^r W_{i} Z_i \label{eq:E_infinite}
\end{align}
where the $Z_i$'s are independent copies of $Z$ and are independent of $(W_i)_{1\leqslant i\leqslant r}$. This equation has been studied by Durrett and Liggett in \cite{durrett1983fixed} and by Guivarc'h in \cite{guivarc1990extension}, and is closely related to implicit renewal theory (see \cite{Goldie1991ImplicitRenewal}). Besides, multifractal dimensions of $\mu_r^\infty$ were studied by Holley and Waymire in \cite{holley1992multifractal} and by Barral \cite{barral1999moments}.
 
Now for $n\in \N^* \cup \{\infty\}$, the asymptotic behavior of $(Z_r^n)_{r\geqslant 2}$ when $r$ goes to infinity was studied by Liu, Rio, and  Rouault in \cite{lrr,liu2000limit,Rouault04}. In \cite{liu2000limit}, the authors studied $(Z_r^{\infty})_{r\geqslant 2}$ and obtained the law of large numbers and the central limit theorem under minimal assumptions. They also provided large deviation results under strong assumptions on the tail of the variable $W$. 
Let
\begin{equation} \label{eq:Lambda_W}
\Lambda_W(t) \defeq \log \Espe[e^{tW}] \in \intervalleof{-\infty}{\infty}
\end{equation}
be the log-Laplace transform of $W$, and let
\begin{equation} \label{eq:Lambda_W_star}
\Lambda_W^*(x) \defeq \sup_{t \in \R} [t x - \Lambda_W(t)] \in \intervalleff{0}{\infty}
\end{equation}
be the Fenchel-Legendre transform of $\Lambda_W$. 
The following result is proved in \cite{lrr,liu2000limit}.

\begin{thm}[light-tailed -- large deviations] \label{th:LDlight}
Let $n \in \N^* \cup \{\infty\}$. Assume that one of the two following statements is true:
\begin{enumerate}
\item \label{th:LDlight_finite} (finite tree) $n \in \N^*$ and $\Lambda_W < \infty$ everywhere;
\item \label{th:LDlight_infinite} (infinite tree) $n = \infty$ and $\esssup(W) < \infty$.
\end{enumerate}
Then, the sequence $(Z_r^n)_{r \geqslant 2}$ satisfies a large deviation principle at speed $(r)_{r \geqslant 2}$ with rate function $\Lambda_W^*$.
\end{thm}

It is worth noticing that, for unbounded random variables $W$, even when the Laplace transform was finite everywhere, the large deviation principle for $n=\infty$ (infinite tree) was unknown.

In view of Lemma \ref{lem:dev2ldp} below, Theorem \ref{th:LDlight} follows from the estimation of left and right deviations. Since the variables $Z_r^n$ are nonnegative, left deviations easily follow from Gärtner-Ellis theorem in great generality, as shown in \cite{liu2000limit}.

\begin{prop}[left large deviations] \label{prop:leftLD}
Let $n \in \N^* \cup \{\infty\}$. If $n = \infty$, we assume that $\Espe[W \log(W)] < \infty$. For all $a \in \intervalleff{0}{1}$,
\[
\frac{1}{r} \log \Prob(Z_r^n \leqslant a)  \xrightarrow[r \to \infty]{} -\Lambda_W^*(a) .
\]
\end{prop}

However, the study of right deviations is more delicate. The case $\esssup(W) < \infty$ is the case where the variables $Z_r^n$ are bounded above (by $\esssup(W)$) and Gärtner-Ellis theorem applies. However, if $\esssup(W) = \infty$, even if $\Lambda_W < \infty$, $\Espe[e^{t Z_r^n}]$ may be infinite for all $t > 0$ and Gärtner-Ellis theorem cannot provide the result in that case.

In this work, we extend these results to the case where $\Lambda_W$ might only be finite on $\intervalleoo{-\infty}{t_0}$ or $\intervalleof{-\infty}{t_0}$ for some $t_0\in \intervalleof{0}{\infty}$ and we obtain,  for both finite and infinite trees,  moderate, large, and very large deviation principles (see Theorems \ref{th:finite} and \ref{th:infinite}). 
In \cite{lrr}, the key argument to obtain the large deviation result in the finite tree and with $\Lambda_W < \infty$ everywhere (Theorem \ref{th:LDlight} with assumption 1)\ is to truncate the variable and use exponential approximation. In that case, the rate function is $\Lambda_W^*$: the large deviations are the same as those of the first level of the tree. 
In the case where $\Lambda_W(t) = \infty$ for $t\geqslant t_0$, the rate function is not $\Lambda_W^*$ as soon as $n\in (\N^*\setminus \{1\}) \cup \{\infty\}$, so this argument seems hopeless. Moreover, the Laplace transforms of the random variables $Z_r^n$ are infinite on $\intervalleoo{0}{\infty}$, hence the exponential Markov's inequality cannot either be used here to prove the upper bounds. Note that, in the case where $\Lambda_W < \infty$ everywhere, we do not know if the truncation argument works. Here, we use different techniques to bypass these problems. In the case of the finite tree, we decompose the event of deviations and control each term; this technique is standard to obtain upper bounds of large deviations for heavy-tailed random variables (see e.g., \cite{Nagaev69-1,Nagaev69-2}). In the case of the infinite tree, we bound the moments of $Z_r^\infty$. For moderate and very large deviations, the upper bound follows immediately from the moment Markov's inequality. As for large deviations, we prove that (see Corollary \ref{cor:moments})
\[
\limsup_{r \to \infty} \frac{1}{r} \log \Espe[(Z_r^\infty)^{\eta r}] = O(\eta^2)
\]
which is the key argument to prove that the rate function is non degenerated and is indeed the limit of the rate functions in finite trees (compare items \ref{th:finiteLD} of Theorems \ref{th:finite} and \ref{th:infinite}).

The paper is organized as follows. The main Theorems for both finite and infinite trees are presented in Section \ref{sec:main}. Section \ref{sec:prooffinite} is dedicated to the proof of the large deviation results for finite trees whereas the proofs for the infinite tree are given in Section \ref{sec:proofinfinite}.

\section{Main results}\label{sec:main}

%\subsection{Specific setting of this article}

From now on, it is always assumed that there exists $c \in \intervalleof{0}{\infty}$ such that
\begin{equation} \label{eq:tail_W}
\frac{1}{w} \log \Prob(W \geqslant w)  \xrightarrow[w\to\infty]{} - c .
\end{equation}

%\begin{itemize}
%\item This assumption entails that $W$ is not a constant.
Since we want to obtain large deviation principles, it is natural from the proofs to assume that $\log \Prob(W \geqslant rw)/r$ converges for all $w \geqslant 0$. Then the limit $\psi(w)$ is either linear in $w$ or infinite, since $\psi(w) = w \psi(1)$ for all nonnegative rational number $w$ and $\psi$ is nonincreasing.

Remark that, by Proposition \ref{prop:Lambda}, item \ref{prop:Lambda_dom} below, $c = \infty$ if and only if $\Lambda_W < \infty$ everywhere. If $c \in \intervalleoo{0}{\infty}$, then $\Lambda_W$ is finite on a neighborhood of zero. Examples where $c \in \intervalleoo{0}{\infty}$ include the cases where the law of $W$ is the exponential distribution of mean $1$ ($c = 1$), any gamma distribution of mean $1$ for instance.

%The case $c = 0$ is the case when $\Lambda_W(t) = +\infty$ for all $t > 0$ ?????

\begin{thm}[finite tree] \leavevmode \label{th:finite}

\begin{enumerate}
\item \label{th:finiteMD} \emph{Moderate deviations.}
For all $\alpha \in \intervalleoo{0}{1/2}$, for all $n \in \N^*$, the sequence $(r^\alpha (Z_r^n - 1))_{r \geqslant 2}$ satisfies a large deviation principle at speed $(r^{1-2\alpha})_{r \geqslant 2}$ with rate function
\[
J(a) \defeq \frac{a^2}{2 \Var(W)} .
\]

\item \label{th:finiteLD} \emph{Large deviations.}
For any $n \in \N^*$, the sequence $(Z_r^n)_{r\geqslant 2}$ satisfies a large deviation principle at speed $(r)_{r \geqslant 2}$ with rate function $I^n$ defined by induction by $I^1=\Lambda_W^*$ and, for all $a \in \R$,
\begin{align}\label{eq:Inrecurrence}
I^{n}(a)
 \defeq \inf \ensavec{cw + I^{n-1}(z)+\Lambda_W^*(s)}{w \geqslant 0,\ z \geqslant 0,\ s \geqslant 0,\ wz + s = a} .
\end{align}
%\begin{align}\label{eq:Inrecurrence}
%I^{n}(a)
% \defeq \begin{cases}
% \Lambda_W^*(a) & \text{if $a \leqslant 1$} \\
% \inf \ensavec{cw + I^{n-1}(z)+\Lambda_W^*(s)}{w \geqslant 0,\ z \geqslant 0,\ s \geqslant 0,\ wz + s \geqslant a} & \text{otherwise.}
%\end{cases}
%\end{align}

\item \label{th:finiteVL} \emph{Very large deviations.}
For all $\alpha >0$, for all $n \in \N^*$, the sequence $(r^{-\alpha} Z_r^n)_{r \geqslant 2}$ satisfies a large deviation principle at speed $(r^{1+\alpha/n})_{r \geqslant 2}$ with rate function $a \mapsto \infty \indic_{a < 0} + c n  a^{1/n} \indic_{a \geqslant 0}$.
\end{enumerate}
\end{thm}

Here are some remarks concerning the very large deviations (Theorem \ref{th:finite}, item \ref{th:finiteLD}). It follows from the proof that the left deviations are in fact at speed $(r)_{r \geqslant 2}$, with rate function $a \mapsto \infty \indic_{a < 0} - \log (p) \indic_{a = 0}$. Concerning the right deviations, for $c = \infty$, the rate function at speed $(r^{1+\alpha/n})_{r \geqslant 2}$ is degenerate, and the informative speed should be $o(r^{1+\alpha/n})$. For instance, if $\log \Prob(W \geqslant w) \sim - c' w^\tau$ with $\tau > 1$, then the speed may depend on $(c', \tau)$, but this result is out of the scope of this paper.

Concerning the large deviations (Theorem \ref{th:finite}, item \ref{th:finiteLD}), according to Proposition \ref{prop:leftLD}, that, for all $n \in \N^*$,
\[
\forall a \leqslant 1 \quad I^n(a) = \Lambda_W^*(a) .
\]
%when $c=0$, for all $n \in \N^*$, $I^n=0$ (noninformative large deviation result), and
% \thierry{22/09 : Ajouter ici aussi l'expression pour $a \geqslant 1$ :
% \[
% I_n(a) = \inf \ensavec{cw + I^{n-1}(z)+\Lambda_W^*(s)}{w \geqslant 0,\ z \in \intervalleff{1}{a},\ s \in \intervalleff{1}{a},\ wz + s = a}.
% \]}
% \pierre{22/09 : The expression of the rate function suggests that $(Z_r^n)_{r \geqslant 2}$ could be exponentially equivalent to $(\tilde{Z}_r^n)_{r \geqslant 2} = ((W_1 Z_{r,1}^{n-1} + W_2 + \dots + W_r) / r)_{r \geqslant 2}$. It is not so. Indeed, for all $\delta > 0$,
% \begin{align*}
% \frac{1}{r} \log \Prob(\lvert Z_r^n - \tilde{Z}_r^n\rvert > \delta)
%  & \geqslant \frac{1}{r} \log \Prob\biggl( \frac{1}{r} \sum_{i=2}^r W_i(Z_{r,i}^{n-1} - 1) > \delta \biggr) \\
%  & \geqslant \frac{1}{r} \log \Prob(W_1 > r \delta^{1/2},\ Z_{r,1}^{n-1} - 1 > \delta^{1/2},\ \forall i \in \intervallentff{2}{r}\ Z_{r, i}^{n-1} - 1 \geqslant 0) \\
%  & \to - (\delta^{1/2} + I^{n-1}(1 + \delta^{1/2}) + \log(2)).
% \end{align*}}
Moreover, if $c = \infty$, then, for all $n \in \N^*$, $I^n=\Lambda_W^*$. This case (equivalent to $\Lambda_W < \infty$ everywhere; see Proposition \ref{prop:Lambda}, item \ref{prop:Lambda_dom}) was studied in \cite{lrr} and the proof relies on exponential approximation. Here we provide another proof that encompasses all the cases $c \in \intervalleof{0}{\infty}$. Moreover, some properties of the functions $I^n$ are given in Propositions \ref{prop:Indecroissante1} and \ref{prop:Indecroissante2}. Notably, we show that $I^n(a) \sim c n a^{1/n}$ as $a \to \infty$, so the rate function of the very large deviations coincides with the asymptotics of the rate function of the large deviations. Also, we prove that $(I^n)_{n \geqslant 1}$ is a decreasing sequence of functions. In particular, we may introduce
\[
I^\infty \defeq \lim_{n \to \infty} \downarrow I^n ,
\]
which appears to be the rate function of the large deviation principle in the infinite tree, as stated below.

%\open{Concerning (iii), if $c=\infty$, then the rate function is infinite except in $0$ so the correct speed should be greater.}

\begin{thm}[infinite tree] \leavevmode \label{th:infinite}

\begin{enumerate}
\item \label{th:infiniteMD} \emph{Moderate deviations.}
For all $\alpha \in \intervalleoo{0}{1/2}$, the sequence $(r^\alpha (Z_r^\infty - 1))_{r \geqslant 2}$ satisfies a large deviation principle at speed $(r^{1-2\alpha})_{r \geqslant 2}$ with rate function $J$.

\item \label{th:infiniteLD} \emph{Large deviations.}
The sequence $(Z_r^\infty)_{r \geqslant 2}$ satisfies a large deviation principle at speed $(r)_{r \geqslant 2}$ with rate function $I^\infty$.

\item \label{th:infiniteVL} \emph{Very large deviations.}
For all $\alpha >0$, the sequence $(r^{-\alpha} Z_r^\infty)_{r \geqslant 2}$ satisfies a large deviation principle at speed $(r \log(r))_{r \geqslant 2}$ with rate function $a \mapsto \infty \indic_{a < 0} + c \alpha e \indic_{a > 0}$.
\end{enumerate}
\end{thm}

\begin{remark}
Here are some remarks concerning the very large deviations in Theorem \ref{th:infinite}.
\begin{enumerate}
\item It follows from the proof that the left deviations are in fact at speed $(r)_{r \geqslant 2}$, with rate function $a \mapsto \infty \indic_{a < 0} - \log (p) \indic_{a = 0}$.

\item Concerning the right deviations, for $c = \infty$, the rate function at speed $(r \log(r))_{r \geqslant 2}$ is degenerate, and the informative speed should be $o(r \log(r))$. For instance, if $\log \Prob(W \geqslant w) \sim - c' w^\tau$ with $\tau > 1$, then the speed may depend on $(c', \tau)$, but this result is out of the scope of this paper.

\item As for $c \in \intervalleoo{0}{\infty}$, the rate function at speed $(r \log(r))_{r \geqslant 2}$ is a positive constant on $\intervalleoo{0}{\infty}$. We expect that, for all $\varepsilon > 0$, $\Loicond{r^{-\alpha} Z_r}{r^{-\alpha} Z_r \geqslant \varepsilon}$ satisfies a large deviation at some speed $o(r \log(r))$ and we conjecture that the speed is in fact $(r)_{r \geqslant 2}$, by approximating the event $\{ r^{-\alpha} Z_r \approx e a / (e - 1) \}$ by the event $\{ W_{1^1} \approx r e a^{1/\log(r^\alpha)}, \dots, W_{1^n} \approx r e a^{1/\log(r^\alpha)} \}$. To get such a result, we should compute the second order in the asymptotics of large deviations, i.e.
\[
\Prob(r^{-\alpha} Z_r \approx a) = \exp(-r\log(r^\alpha) c e + r K(a) + o(r)) .
\]
% Qu'en est-il de $\Loicond{r^{-\alpha} Z_r}{r^{-\alpha} Z_r \geqslant \varepsilon}$ ? Si l'on approxime $\{ r^{-\alpha} Z_r \approx e a / (e - 1) \}$ par l'événement $\{ W_{1^1} \approx r e a^{1/\log(r^\alpha)}, \dots, W_{1^n} \approx r e a^{1/\log(r^\alpha)} \}$, on conjecture
% \[
% \Probcond{r^{-\alpha} Z_r \approx a}{r^{-\alpha} Z_r \geqslant \varepsilon} \approx e^{-r \log(a / \varepsilon)} .
% \]
% Pour arriver à quelque chose, il faudrait obtenir
% \[
% \Prob(r^{-\alpha} Z_r \approx a) = \exp(-r\log(r^\alpha) c e + r I(a) + o(r)) .
% \]
% Pascal suggère de regarder la décomposition épinale, qui permet de faire un changement de mesure pour étudier ce cas où une épine se dégage (voir Lyons 1997).
\end{enumerate}

\end{remark}

The proofs of the large deviation principles derive from left and right deviation estimates and the standard argument that is recalled in appendix (Lemma \ref{lem:dev2ldp}). Table \ref{table:summary} below summarizes the pre-existing results and our contribution:
\begin{table}[h!]
\centering
\newcolumntype{x}[1]{>{\centering\arraybackslash\hspace{0pt}}p{#1}}
\begin{tabular}{|c|x{4cm}|x{4cm}|x{4cm}|}
\cline{2-4}
\multicolumn{1}{c|}{} & $c = \infty$ & $c = \infty$ & $c \in \intervalleoo{0}{\infty}$ \\
\multicolumn{1}{c|}{} & $\esssup(W) < \infty$ & $\esssup(W) = \infty$ & $\esssup(W) = \infty$ \\
\hline
$n \in \N^*$ & \multicolumn{2}{c|}{\cite[Theorem 6.2 (a)]{lrr}} & Theorem \ref{th:finite} \\
\hline
$n = \infty$ & \cite[Theorem 1.4]{liu2000limit} & \multicolumn{2}{c|}{Theorem \ref{th:infinite}} \\
\hline
\end{tabular}
\caption{Summary of previous and new results. Theorems \ref{th:finite} and \ref{th:infinite} encompass the other results on the same line.} \label{table:summary}
\end{table}

\section{Finite tree}\label{sec:prooffinite}

\subsection{About the functions \texorpdfstring{$I^n$}{In}}

The following estimates will be useful in the study of the functions $I^n$.

\begin{prop} \label{prop:Lambda}
Recall the definition of $c$ in \eqref{eq:tail_W}, $\Lambda_W$ in \eqref{eq:Lambda_W}, and $\Lambda_W^*$ in \eqref{eq:Lambda_W_star}.
\begin{enumerate}
\item \label{prop:Lambda_dom} $c = \sup(\dom(\Lambda_W))$; 
\item \label{prop:Lambda_star} If $c \in \intervalleoo{0}{\infty}$, then $\Lambda_W^*(w) \sim cw$ as $w\to \infty$.
\end{enumerate}
\end{prop}

\begin{proof}[Proof of Proposition \ref{prop:Lambda}] \leavevmode

\begin{enumerate}
\item Let $m = \sup(\dom(\Lambda_W))$. First, let $t < c' < c$. By \eqref{eq:tail_W}, there exists $w_0 > 0$ such that, for all $w \geqslant w_0$,
\[
\Prob(W \geqslant w) \leqslant e^{- c' w} .
\]
Since $W \geqslant 0$, one has
\begin{align*}
\Espe[e^{t W}]
 = \int_{u=0}^\infty \Prob(e^{tW} \geqslant u) du 
 = \int_{u=0}^\infty \Prob(W \geqslant \log(u) / t) du 
\leqslant 
%e^{w_0} + \int_{u=e^{w_0}}^\infty e^{- c' \log(u) / t} du \\=
  e^{t w_0} + \int_{u=e^{t w_0}}^\infty u^{- c' / t} du 
 < \infty . 
\end{align*}
Thus $t\leqslant m$, so $c\leqslant m$. Second, let $t<m$. By Markov's inequality, 
\begin{align*}
\frac{1}{w} \log \Prob(W \geqslant w) &\leqslant -t + \frac 1w \log\Espe[e^{tW}].
\end{align*}
Taking the limit superior as $w\to \infty$ leads to $-c\leqslant -t$ whence $c\geqslant m$.

\item First, for all $\varepsilon > 0$,
\[
\Lambda_W^*(w)
 = \sup_{t \in \R} [t w - \Lambda_W(t)]
 \geqslant (c - \varepsilon) w - \Lambda_W(c - \varepsilon) ,
\]
so, since $\Lambda_W(c - \varepsilon) < \infty$ by item \ref{prop:Lambda_dom}, $\liminf_{w \to \infty} \Lambda_W^*(w) / w \geqslant c - \varepsilon$. Since $\varepsilon > 0$ is arbitrary, we get $\liminf_{w \to \infty} \Lambda_W^*(w) / w \geqslant c$.

Secondly, for $w \geqslant 1$, by item \ref{prop:Lambda_dom},
\[
\Lambda_W^*(w)
 = \sup_{0 \leqslant t \leqslant c} [t w - \Lambda_W(t)]
 \leqslant c w - \inf_{0 \leqslant t \leqslant c} \Lambda_W(t)
 = c w ,
\]
so $\limsup_{w \to \infty} \Lambda_W^*(w) / w \leqslant c$.\qedhere
\end{enumerate}
\end{proof}

% I^n

Now we turn to the properties of the functions $I^n$.

\begin{prop} \leavevmode \label{prop:Indecroissante1}
\begin{enumerate}
\item \label{prop:Indecroissante_1} $I^n(1)=\Lambda_W^*(1)=0$.
\item \label{prop:Indecroissante_increasing} For all $n \in \N^*$, the function $I^n$ is nondecreasing on $\intervallefo{1}{\infty}$. Moreover, if $c \in \intervalleof{0}{\infty}$, then the function $I^n$ is increasing on $\intervallefo{1}{\infty} \cap \{ I^n < \infty \}$.
\item \label{prop:Indecroissante_positive}If $c \in \intervalleof{0}{\infty}$, then, for all $a>1$, $I^n(a)>0$.
\item \label{prop:Indecroissante_nonincreasing} The sequence of functions $(I^n)_{n\geqslant 1}$ is nonincreasing. 
\item \label{prop:Indecroissante_recurrence2} For $a\geqslant 1$,
\begin{align}\label{eq:Inrecurrence2}
I^n(a) = \inf
\ensavec{cw + I^{n-1}(z)+\Lambda_W^*(s)}{w \geqslant 0,\ z \in \intervalleff{1}{a},\ s \in \intervalleff{1}{a},\ wz + s = a} .
\end{align}
\end{enumerate}
\end{prop}

\begin{proof}[Proof of Proposition \ref{prop:Indecroissante1}] \leavevmode

\begin{enumerate}
\item Obvious. 

\item Let us define the function
\begin{align}\label{def:Gn}
G^{n-1}(w,z,s) \defeq cw + I^{n-1}(z)+\Lambda_W^*(s).
\end{align}
Let $a_1 > a_0 \geqslant 1$ and $(w, z, s) \in (\R_+)^3$ be such that $wz + s = a_1$. Either $s\geqslant a_0$, then 
$G^{n-1}(w,z,s) \geqslant G^{n-1}(0,z,a_0)$, because $\Lambda_W^*$ is nondecreasing on $\intervallefo{1}{\infty}$; or $s<a_0$, then 
$G^{n-1}(w,z,s) \geqslant G^{n-1}((a_0-s)/z,z,s)$. Hence $I^n(a_1) \geqslant I^n(a_0)$. If $c \in \intervalleoo{0}{\infty}$, then the preceding argument holds with strict inequalities, using the fact that $\Lambda_W^*$ (resp.\ $w \mapsto cw$) is increasing on $\intervallefo{1}{\infty}$ (resp.\ on $\intervallefo{0}{\infty}$), so $I^n$ is increasing on $\intervallefo{1}{\infty}$. If $c = \infty$, then $I^n = \Lambda_W^*$ is increasing on $\intervallefo{1}{\infty} \cap \{ \Lambda_W^* < \infty \}$.

\item It is an immediate consequence of items \ref{prop:Indecroissante_1} and \ref{prop:Indecroissante_increasing}.
%We proceed by induction. Since $I^1=\Lambda_W^*$ it is true for $n=1$. Now we assume that it is  true for some $n\geqslant1$. Let $a>1$ by  \eqref{eq:Inrecurrence} 
%\begin{align*}
%I^{n+1}(a)&\geqslant \min (\inf\ensavec{cw+\Lambda_W^*(s)}{w \geqslant 0,\ 0\leqslant z \leqslant 2,\ s \geqslant 0,\ wz + s = a},\\&\hspace{2cm}\inf\ensavec{cw+I^n(2)+\Lambda_W^*(s)}{w \geqslant 0,\  z>2 ,\ s \geqslant 0,\ wz + s = a})\\
%&
%\geqslant \min \left(\inf\ensavec{\frac{c(a-s)}{2} w+\Lambda_W^*(s)}{ s \geqslant 0},I^n(2)\right)\\
%&>0.
%\end{align*}

\item First, we prove that $I^2\leqslant I^1$. It is enough to take $(w,z,s)=(0,1,a)$ in \eqref{eq:Inrecurrence} to get
\[
I^2(a) \leqslant \Lambda_W^*(a)=I^1(a)
\] 
by item \ref{prop:Indecroissante_1} and $I^1=\Lambda^*$. We conclude by induction:
\begin{align*}
I^{n+1}(a) 
%&= \inf
%\ensavec{cw + I^n(z)+\Lambda_W^*(s)}{w \geqslant 0,\ z \geqslant 0,\ s \geqslant 0,\ wz + s = a}\\
&\leqslant \inf
\ensavec{cw + I^{n-1}(z)+\Lambda_W^*(s)}{w \geqslant 0,\ z \geqslant 0,\ s \geqslant 0,\ wz + s = a}= I^n(a).
\end{align*}

\item For all $(w, z, s) \in (\R_+)^3$, $G^{n-1}(w,z,s) \geqslant G^{n-1}(w,\max(1,z),\max(1,s))$.
So, for all $a\geqslant 1$,
\begin{align*}
I^{n}(a)
 & = \inf
\ensavec{cw + I^{n-1}(z)+\Lambda_W^*(s)}{w \geqslant 0,\ z \geqslant 1,\ s \geqslant 1,\ wz + s \geqslant a} \\
 & = \inf
\ensavec{cw + I^{n-1}(z)+\Lambda_W^*(s)}{w \geqslant 0,\ z \geqslant 1,\ s \geqslant 1,\ wz + s = a} ,
\end{align*}
proceeding as in the proof of item \ref{prop:Indecroissante_increasing}.
%Now, if $wz + s > a$, either $s\geqslant a$, then 
%$G^{n-1}(w,z,s) \geqslant G^{n-1}(0,z,a)$; or $s<a$, then 
%$G^{n-1}(w,z,s) \geqslant G^{n-1}((a-s)/z,z,s)$.
Hence, since $w z \geqslant 0$ and $w z + s = a$ imply $s \leqslant a$,
\begin{align*}
I^{n}(a)
 & = \inf
\ensavec{cw + I^{n-1}(z)+\Lambda_W^*(s)}{w \geqslant 0,\ z \geqslant 1,\ s \in \intervalleff{1}{a},\ wz + s = a} \\
 & = \inf
\ensavec{cw + I^{n-1}(z)+\Lambda_W^*(s)}{w \geqslant 0,\ z \in \intervalleff{1}{a},\ s \in \intervalleff{1}{a},\ wz + s = a} ,
\end{align*}
where the latter equality is obvious for $c = 0$ and follows, for $c \in \intervalleof{0}{\infty}$, from the fact that $I^{n-1}(z) > I^{n-1}(a) \geqslant I^n(a)$ for all $z > a$ (by items \ref{prop:Indecroissante_increasing} and \ref{prop:Indecroissante_nonincreasing}), and $cw + \Lambda_W^*(s) > 0$ for all $(w, s) \in \intervallefo{0}{\infty} \times \intervallefo{1}{\infty} \setminus \{ (0, 1) \}$.\qedhere

%As a consequence, we also have, for $a\geqslant 1$, 
%\begin{align*}
%I^{n}(a) = \inf
%\ensavec{cw + I^{n-1}(z)+\Lambda_W^*(s)}{w \geqslant 0,\ z \geqslant 0,\ s \geqslant 0,\ wz + s = a}
%\end{align*}
%since the latter infimum is between that in %\eqref{eq:Inrecurrence} and that in \eqref{eq:Inrecurrence2}.

%Finally, Equation \eqref{eq:Inrecurrence1strict} is a straightforward consequence of \eqref{eq:Inrecurrence} and $G^{n-1}(w, 0, a/(1-\varepsilon)) \geqslant G^{n-1}(0, 0, a/(1-\varepsilon)) \geqslant G^{n-1}(0, 1, a/(1-\varepsilon))$.\qedhere
%, and the continuity of $w \mapsto c w$ and $\Lambda_W^*$.
\end{enumerate}
\end{proof}

For all $a \geqslant 1$ and $z \in \intervalleff{1}{a}$, let
\begin{align} \label{eq:def_h}
h(a, z)
 & \defeq \frac{c a}{z} - \sup_{1 \leqslant s \leqslant a} \Bigl[ \frac{c s}{z} - \Lambda_W^*(s) \Bigr]  
\end{align}
so that, using \eqref{eq:Inrecurrence2} (Proposition \ref{prop:Indecroissante_1}, item \ref{prop:Indecroissante_recurrence2}),
\begin{align}\label{eq:Inh}
I^n(a) = \inf_{1 \leqslant z \leqslant a} \bigl[ I^{n-1}(z) + h(a, z) \bigr] .
\end{align}
Let also
\[
a^* \defeq \inf \ensavec{s \geqslant 0}{\Lambda_W'(s) = a}
\]
be the convex conjugate of $a$ through $\Lambda_W$. Note that, if $a \leqslant \Lambda_W'(c^-)$, then $a^*$ is the unique solution in $s$ of $\Lambda_W'(s) = a$; and if $a > \Lambda_W'(c^-)$, then $a^* = \infty$.

\begin{lem}\label{lem:h}
For all $a \geqslant 1$ and $z \in \intervalleff{1}{a}$,
\begin{align*}
h(a, z)
 & = \begin{cases}
\Lambda_W^*(a) & \text{if $z \leqslant c / a^*$} \\
\frac{c a}{z} - \Lambda_W\bigl( \frac{c}{z} \bigr) & \text{if $z \geqslant c / a^*$.}
\end{cases}
\end{align*}
and the function $h$ is continuous.
\end{lem}

\begin{proof}[Proof of Lemma \ref{lem:h}]
An easy computation of Legendre transform leads to
\begin{align*}
h(a, z)
 & = \frac{c a}{z} - \sup_{1 \leqslant s \leqslant a} \Bigl[ \frac{c s}{z} - \Lambda_W^*(s) \Bigr] \\
 & = \begin{cases}
\frac{c a}{z} - \bigl( \Lambda_W(a^*) + \bigl( \frac{c}{z} - a^* \bigr) a \bigr) & \text{if $c / z \geqslant a^*$} \\
\frac{c a}{z} - \Lambda_W\bigl( \frac{c}{z} \bigr) & \text{if $c / z \leqslant a^*$.}
\end{cases} \\
 & = \begin{cases}
\Lambda_W^*(a) & \text{if $z \leqslant c / a^*$} \\
\frac{c a}{z} - \Lambda_W\bigl( \frac{c}{z} \bigr) & \text{if $z \geqslant c / a^*$.}
\end{cases}
\end{align*}
 
Let $A \geqslant 1$. Since $H \colon (a, z, w)\mapsto cw + \Lambda_W^*(a-wz)$ is uniformly continuous on the compact set $\ensavec{(a, z, w)}{1 \leqslant a \leqslant A,\ 1 \leqslant z \leqslant a,\ 0 \leqslant w \leqslant \frac{a-1}{z}}$ and $(a, z) \mapsto (a-1) / z$ is continuous on $D(A) \defeq \ensavec{(a, z)}{1 \leqslant a \leqslant A,\ 1 \leqslant z \leqslant a}$, $h$ is continuous on $D(A)$.
\end{proof}

\begin{prop} \leavevmode \label{prop:Indecroissante2}
\begin{enumerate}
\item \label{prop:Indecroissante_continuous} For all $n \in \N^*$, the function $I^n$ is continuous.
\item \label{prop:Indecroissante_equiv} For all $n \in \N^*$, $I^n(a)\sim cna^{1/n}$ as $a \to \infty$. 
\end{enumerate}
\end{prop}

\begin{proof}[Proof of Proposition \ref{prop:Indecroissante2}] \leavevmode

\begin{enumerate}
\item Let $a_0 \geqslant 1$ and $\varepsilon > 0$. Since the function $I^n$ is nondecreasing by Proposition \ref{prop:Indecroissante1}, item \ref{prop:Indecroissante_increasing}, it suffices to prove that there exists $a_1 > a_0$ such that $I^n(a_1) \leqslant I^n(a_0) + \varepsilon$. Since $h$ is uniformly continuous on the compact set $D(a_0 + 1)$ where $D(a)$ has been defined int the proof of Lemma \ref{lem:h}, there exists $a_1 \in \intervalleoo{a_0}{a_0 + 1}$ such that, for all $z \in \intervalleff{1}{a_0}$, $h(a_1, z) < h(a_0, z) + \varepsilon$. Then, by \eqref{eq:Inh}, 
\begin{align*}
I^n(a_1)
 & = \inf_{1 \leqslant z \leqslant a_1} \Bigl[ I^{n-1}(z) + h(a_1, z) \Bigr] \\ 
 & \leqslant \inf_{1 \leqslant z \leqslant a_0} \Bigl[ I^{n-1}(z) + h(a_1, z) \Bigr]  \\
 & \leqslant \inf_{1 \leqslant z \leqslant a_0} \Bigl[ I^{n-1}(z) + h(a_0, z) \Bigr] + \varepsilon  \\
 & = I^n(a_0) + \varepsilon . 
\end{align*}
\item We proceed by induction. The result holds for $n=1$ by Proposition \ref{prop:Lambda}, item \ref{prop:Lambda_star} and $I^1=\Lambda^*$.
Let $n \in \N^*\setminus \{1\}$. First, notice that, for all $a \geqslant 1$ and $z \in \intervalleff{1}{a}$, $h(a,z)\leqslant ca/z$ (it suffices to take $s=1$ in the supremum in \eqref{eq:def_h}) so that
\begin{align*}
I^{n+1}(a) &= \inf_{1 \leqslant z \leqslant a} \Bigl[ I^n(z) + h(a, z) \Bigr] \\
&\leqslant \inf_{1 \leqslant z \leqslant a} \Bigl[ I^n(z) + \frac{ca}{z} \Bigr] \\
&\leqslant I^n\bigl(a^{n/(n+1)}\bigr) + c a^{1/(n+1)} \\
& \sim c(n+1) a^{1/(n+1)}
\end{align*}
by \eqref{eq:Inh}, taking $1\leqslant z=a^{n/(n+1)}\leqslant a$, and by induction. Let us turn to the minoration. For all $\varepsilon>0$, there exists $A_\varepsilon>0$ such that
\begin{align*}
I^{n+1}(a) & = \inf_{1 \leqslant z \leqslant a} \Bigl[ I^n(z) + h(a, z) \Bigr] \\
 & \geqslant c \inf_{1 \leqslant z \leqslant a} \Bigl[ nz^{1/n} (1-\varepsilon)  + \frac{a}{z} - \sup_{1 \leqslant s \leqslant a} \Bigl( s \Bigl( \frac{1}{z} - (1 - \varepsilon) \Bigr) \Bigr)  \Bigr] - A_\varepsilon \\
 & \geqslant c \inf_{1 \leqslant z \leqslant a} \Bigl[ nz^{1/n} (1-\varepsilon)  + a \min\Bigl(\frac{1}{z} , 1-\varepsilon \Bigr) \Bigr] - A_\varepsilon \\
 & = c \min ((n+1) a^{1/(n+1)}(1-\varepsilon)^{n/(n+1)},(n+a)(1-\varepsilon)) -A_\varepsilon\\
 & \sim c(n+1) a^{1/(n+1)}(1-\varepsilon)^{n/(n+1)}
\end{align*}
using \eqref{eq:Inh}, by induction, and as $a\to \infty$. Since $\varepsilon$ is arbitrary, we get the result. \qedhere
\end{enumerate}
\end{proof}

For all $n \in \N^*$, let us define
\begin{align}\label{def:pt_rupture}
a_n\defeq \inf \ensavec{a\geqslant 1}{I^{n+1}(a)<I^n(a)}.
\end{align}

\begin{prop}\label{prop:pt_rupture}
The sequence $(a_n)_{n\geqslant 1}$ is increasing and diverges to infinity.
\end{prop}

\begin{proof}[Proof of Proposition \ref{prop:pt_rupture}]
Let $n \in \N^* \setminus \{ 1 \}$. For all $a \in \intervalleff{1}{a_{n-1}}$, 
\begin{align*}
I^{n+1}(a)
 & = \inf_{1 \leqslant z \leqslant a} \Bigl[ I^n(z) + h(a, z) \Bigr]  = \inf_{1 \leqslant z \leqslant a} \Bigl[ I^{n-1}(z) + h(a, z) \Bigr] =I^n(a)
\end{align*}
by \eqref{eq:Inh}. Thus $a_{n-1} \leqslant a_n$. Now, assume by contradiction that the sequence $(a_n)_{n \geqslant 1}$ is upper bounded by some $A > 0$. Let $a_W \defeq \inf \ensavec{a \geqslant 1}{a a^* \geqslant c}$. For all $a \geqslant a_W$ and $z \in \intervalleff{1}{a}$,
\begin{align*}
h(a, z)
 & \geqslant h(a, a)
 = c - \Lambda_W(c / a)
 \geqslant c - \Lambda_W(c / a_W)
 = \Lambda_W^*(a_W)
 \eqdef \rho
 > 0 .
\end{align*}

Since the function $I^n$ is uniformly continuous on $\intervalleff{1}{A}$ (by Proposition \ref{prop:Indecroissante2}, item \ref{prop:Indecroissante_continuous}), there exists $\eta(A) \in \intervalleoo{0}{1}$ such that, for all $(a, z) \in \intervalleff{1}{A+1}^2$ with $z \in \intervalleff{a - \eta(A)}{a}$, $I^n(z) > I^n(a) - \rho$ (note that it implies $\eta(A) < a_W - 1$). Then, for all $a \in \intervalleff{a_W}{A+1}$ and $z \in \intervalleff{a-\eta(A)}{a}$,
\begin{align*}
I^n(z) + h(a, z)
 & > (I^n(a) - \rho) + \rho
 = I^n(a)
 \geqslant I^{n+1}(a) .
\end{align*}

Hence, for all $a \leqslant a_{n-1} + \eta(A)$,
\begin{align*}
I^{n+1}(a)
 & = \inf_{1 \leqslant z \leqslant a} \Bigl[ I^n(z) + h(a, z) \Bigr]  & \text{(by \eqref{eq:Inh})}\\
 & = \inf_{1 \leqslant z \leqslant a - \eta(A)} \Bigl[ I^n(z) + h(a, z) \Bigr] \\
 & = \inf_{1 \leqslant z \leqslant a - \eta(A)} \Bigl[ I^{n-1}(z) + h(a, z) \Bigr] & \text{(since $z \leqslant a_{n-1}$)} \\
 & \geqslant I^n(a) ,
\end{align*}
hence $a \leqslant a_n$. So $a_{n-1} + \eta(A) \leqslant a_n$: this is in contradiction with the fact that, for all $n \in \N^*$, $a_n \leqslant A$, and the conclusion of the proposition follows.
\end{proof}

\subsection{Proof of Theorem \ref{th:finite}, item \ref{th:finiteMD} (moderate deviations)}

We proceed by induction over $n \in \N^*$. The result for $n = 1$ stems from the moderate deviation principle for a sum of independent and identically distributed random variables (see \cite[Theorem 3.7.1]{DZ98}). Now, let $n \in \N^*\setminus \{1\}$ be such that $(r^\alpha(Z_r^{n-1} - 1))_{r \geqslant 2}$ satisfies a large deviation principle at speed $(r^{1 - 2 \alpha})_{r \geqslant 2}$ with rate function $J$.

\medskip

\emph{Left deviations} --- Set, for $s < 0$, 
\[
\Lambda_r^n (s) = \frac{1}{r^{1-2\alpha}} \log \mathbb E \left[e^{ s r^{1-\alpha} (Z_r^n-1)} \right] .
\]
We will prove that 
\begin{equation}\label{eq:mod}
\Lambda_r^n (s) \xrightarrow[r \to \infty]{} \frac{s^2\Var(W)}{2} .
\end{equation}
Then, applying a unilateral version of the Gärtner-Ellis theorem (see \cite{PS75} or \cite[Theorem 10]{FATP2020}), we get
\[
\frac{1}{r^{1-2\alpha}} \log \Prob( r^\alpha (Z_r^n - 1) \leqslant - a)
 \xrightarrow[r \to \infty]{} - \frac{a^2}{2 \Var(W)} .
\]
By \eqref{eq:defZrnintro}, we have 
\begin{align}
Z_r^n
 & = \frac{1}{r} \sum_{i=1}^r W_{i} Z_{r,i}^{n-1}, \label{eq:E_finite}
\end{align}
where 
\begin{align*}
Z_{r,i}^{n-1}\defeq \frac{1}{r^{n-1}} \sum_{1\leqslant i_2, \dots, i_n\leqslant r} 
 W_{i,i_2} \cdots W_{i,i_2,\dots, i_n}
\end{align*}
is distributed as $Z_r^{n-1}$. By independence and identity of distributions, we have
\begin{align*}
\Lambda_r^n (s)
 & = r^{2\alpha-1} \log \Espe\left[e^{ s r^{-\alpha}\sum_{i=1}^r (W_i Z_{r,i}^{n-1}-1)}\right] 
 = r^{2\alpha} \log \Espe\left[e^{ s r^{-\alpha} (W Z_{r}^{n-1}-1)}\right] .
\end{align*}
Using the Taylor expansion $e^x = 1 + x + x^2 e^{\theta(x)} / 2$ with $\theta(x) \leqslant \max(x, 0)$ and the fact that $\Espe[W Z_r^n] = \Espe[W] \Espe[Z_r^n] = 1$, we get
\begin{align*}
\Lambda_r^n(s)
 & = r^{2\alpha} \log\Bigl(1 + \frac{s^2}{2 r^{2 \alpha}} \Espe\bigl[ (W Z_r^n - 1)^2 e^{\theta(s r^{-\alpha} (W Z_r^n - 1))} \bigr] \Bigr)
 \xrightarrow[r \to \infty]{} \frac{s^2}{2} \Espe[(W - 1)^2] ,
\end{align*}
which is \eqref{eq:mod}. Indeed,
\begin{align*}
\Bigl| \Espe\bigl[ & (W Z_r^n - 1)^2
  e^{\theta(s r^{-\alpha}  (W Z_r^n - 1))} \bigr] - \Espe[(W - 1)^2] \Bigr| \\
 & \leqslant \abs{\Espe\bigl[ (W Z_r^n - 1)^2 \bigl( e^{\theta(s r^{-\alpha} (W Z_r^n - 1))} - 1 \bigr) \bigr]} + \abs{\Espe\bigl[ (W Z_r^n - 1)^2 - (W - 1)^2 \bigr]} \\
 & \leqslant \abs{\Espe\bigl[ (W Z_r^n - 1)^2 \bigr]} \bigl( e^{- s r^{-\alpha}} - 1 \bigr) + \abs{\Espe\bigl[ (W Z_r^n - 1)^2 - (W - 1)^2 \bigr]}  \xrightarrow[r \to \infty]{} 0 ,
\end{align*}
since
\begin{align*} %\label{eq:moments_Zr2}
\Espe[(Z_r^{n-1})^2]
 & = \Espe\biggl[ \biggl( \frac{1}{r} \sum_{i=1}^r W_{i} Z_{r,i}^{n-2} \biggr)^2 \biggr] 
 = \frac{1}{r^2}\bigl( r \Espe[W^2] \Espe[(Z_r^{n-2})^2] + r (r - 1) \Espe[W]^2 \Espe[Z_r^{n-2}]^2 \bigr) 
  \xrightarrow[r \to \infty]{} 1 .
\end{align*}

\begin{remark}
Obviously, the previous estimate generalizes and the multinomial expansion yields
\begin{equation} \label{eq:moments_Zrn}
\forall h \in \N \quad \Espe[(Z_r^{n-1})^h] \xrightarrow[r \to \infty]{} 1 .
\end{equation}
Controlling the moments of $Z_r^n$ appears to be crucial in the study of the right deviations for $n = \infty$ (see Section \ref{sec:moments_Zr}). Here, mimicking the proof of the large deviations, we resort to induction to obtain the deviations in the finite tree.
%%%%%%%%%%%%%%%%%%%%%
% Autre preuve
%%%%%%%%%%%%%%%%%%%%%
%\begin{prop}\label{prop:moments_Zrn}
%For all $n\in \N^*$ and $h\in \N$, $\Espe[(Z_r^n)^h]\to 1$ as $r \to \infty$.
%\end{prop}
%
%
%\begin{proof}[Proof of Proposition \ref{prop:moments_Zrn}]
% By Lemma \ref{lem:convexity} item (i), one has
%\begin{align*}
%\Espe[(Z_r^n)^h] & = \frac{1}{r^h} \sum_{h_1+\dots+h_r=h}
%\binom{h}{h_1,\dots,h_r} \prod_{j=1}^r \Espe[ (WZ_{r}^{n-1} )^{h_j}]\leqslant  \Espe[ (WZ_{r}^{n-1})^h ]
%\end{align*}
%which is bounded by induction on $n$. The conclusion follows since $Z_r^n\to 1$ almost surely as $r\to \infty$ and applying \cite[Example 2.21]{van1998asymptotic}. 
%\end{proof}
\end{remark}

\medskip

\emph{Right deviations} --- Let us turn to the right deviations. Let $a \geqslant 0$. Using \eqref{eq:E_finite}, one gets
\[
\Prob(r^\alpha (Z_r^n - 1) \geqslant a)
 = \Prob\biggl( \sum_{i=1}^r (W_i Z_{r,i}^{n-1} - 1) \geqslant a r^{1-\alpha} \biggr) .
\]
Let us prove the lower bound. For all $\varepsilon > 0$,
\begin{align*}
\frac{1}{r^{1-2\alpha}} \log &\ \Prob\biggl( \sum_{i=1}^r (W_i Z_{r,i}^{n-1} - 1) \geqslant a r^{1-\alpha} \biggr) \\
 & \geqslant \frac{1}{r^{1-2\alpha}} \log \Prob\biggl( \sum_{i=1}^r (W_i Z_{r,i}^{n-1} - 1) \geqslant a r^{1-\alpha},\ \forall i \in \intervallentff{1}{r}\ Z_{r,i}^{n-1} \geqslant 1 - \varepsilon r^{-\alpha} \biggr) \\
 & \geqslant \frac{1}{r^{1-2\alpha}} \log \Prob\biggl( \sum_{i=1}^r (W_i (1 - \varepsilon r^{-\alpha}) - 1) \geqslant a r^{1-\alpha} \biggr) +  r^{2\alpha} \log \Prob(Z_r^{n-1} \geqslant 1 - \varepsilon r^{-\alpha}) .
\end{align*}
By induction (left deviations for $Z_r^{n-1}$),
\[
\frac{1}{r^{1-2\alpha}} \log \Prob(Z_r^{n-1} - 1 < - \varepsilon r^{-\alpha})
 \xrightarrow[r \to \infty]{} - J(-\varepsilon)
 < 0 ,
\]
hence
\[
r^{2\alpha} \log \Prob(Z_r^{n-1} \geqslant 1 - \varepsilon r^{-\alpha})
 \xrightarrow[r \to \infty]{} 0 .
\]
Now, since $\Espe[e^{sW}]<\infty$ for all $s<c$ (Proposition \ref{prop:Lambda}, item \ref{prop:Lambda_dom}), we can proceed as in the proof of \cite[Theorem 3.1.1]{DZ98} to apply the unilateral version of Gärtner-Ellis theorem and get
\begin{align}\label{eq:modW}
\frac{1}{r^{1-2\alpha}} \log \Prob\biggl( \sum_{i=1}^r (W_i (1 - \varepsilon r^{-\alpha}) - 1) \geqslant a r^{1-\alpha} \biggr)
  \xrightarrow[r \to \infty]{} - J(a + \varepsilon) .
\end{align}
Letting $\varepsilon \to 0$, one finally gets
\begin{align}
\liminf_{r \to \infty} \frac{1}{r^{1-2\alpha}} \log \Prob(r^\alpha (Z_r^n - 1) \geqslant a)
 \geqslant - J(a) .\label{eq:moderate_fini_minoration}
\end{align}

Let us turn to the upper bound. Let $\varepsilon > 0$. For all $r \in \N^* \setminus \{ 1 \}$ and for all $q\in \intervallentff{0}{r}$, we define
\begin{align*}
P_{n,r,q}
 & = \Prob\biggl( \sum_{i=1}^r (W_i Z_{r,i}^{n-1} - 1) \geqslant a r^{1-\alpha},\ \forall i \in \intervallentff{1}{q}\  Z_{r,i}^{n-1} - 1\geqslant \varepsilon r^{-\alpha}, \\
 & \hspace{6.5cm} \forall i \in \intervallentff{q+1}{r}\ Z_{r,i}^{n-1} - 1 < \varepsilon r^{-\alpha} \biggr).
\end{align*}
First, for $q_0\in \intervallentff{0}{r}$ and $\varepsilon' \in \intervalleoo{0}{\varepsilon}$, one has by induction
\begin{align}
\sum_{q=q_0}^r \binom{r}{q} P_{n,r,q}
 &\leqslant \sum_{q=q_0}^r \binom{r}{q} \Prob(Z_r^{n-1} - 1 \geqslant \varepsilon r^{-\alpha})^q\notag \\
 &\leqslant \sum_{q=q_0}^r \left(re^{-r^{1-2\alpha} J(\varepsilon')}\right)^q \notag\\
 &\leqslant \frac{\left(re^{-r^{1-2\alpha}J(\varepsilon')}\right)^{q_0}}{1-re^{-r^{1-2\alpha} J(\varepsilon')}} \notag\\
 &\leqslant e^{-r^{1-2\alpha} J(a)} ,\label{eq:moderate_fini_Pnrq0}
\end{align}
as soon as $q_0 J(\varepsilon') > J(a)$ and $r$ is large enough.
Second, for $q\in \intervallentff{0}{q_0-1}$, one has
\begin{align*}
\frac{1}{r^{1-2\alpha}} \log(P_{n, r, q})
 & \leqslant \frac{1}{r^{1-2\alpha}} \log \Prob\biggl( \sum_{i=1}^q (W_i Z_{r,i}^{n-1} - 1) + \sum_{i=q+1}^r (W_i (1 + \varepsilon r^{-\alpha}) - 1) \geqslant a r^{1-\alpha} \biggr). 
 %\label{eq:fini_moderate_Pnrq_0} 
\end{align*}
Now, for all $b \geqslant 0$, considering whether $r^{\alpha/2}(Z_r^{n-1} - 1)$ is larger than $1$ or not,
\begin{align*}
\Prob(W Z_r^{n-1} - 1 \geqslant b r^{1-\alpha})
 & = \Prob\biggl( \frac{W}{r^{1-\alpha/2}} r^{\alpha/2}(Z_r^{n-1} - 1) + \frac{W - 1}{r^{1-\alpha}} \geqslant b \biggr) \\
% & \leqslant \Prob\bigl( r^{\alpha/2}(Z_r^{n-1} - 1) \geqslant 1 \bigr) + \Prob\biggl( \frac{W}{r^{1-\alpha/2}} + \frac{W - 1}{r^{1-\alpha}} \geqslant b \biggr) \\
 & \leqslant \Prob\bigl( r^{\alpha/2}(Z_r^{n-1} - 1) \geqslant 1 \bigr) + \Prob\biggl( \frac{W}{r^{1-\alpha}} \geqslant \frac{b + r^{-(1-\alpha)}}{1 + r^{-\alpha/2}} \biggr),
\end{align*}
whence, by induction, using the large deviation principle for $(r^{\alpha/2}(Z_r^{n-1} - 1))_{r \geqslant 2}$, and assumption \eqref{eq:tail_W} on the tail of $W$, we get
\begin{equation}
\frac{1}{r^{1-2\alpha}} \log \Prob(W Z_r^{n-1} - 1 \geqslant b r^{1-\alpha}) \\
 \xrightarrow[r \to \infty]{} \begin{cases} 0 & \text{if $b = 0$} \\ -\infty & \text{otherwise.} \end{cases} \label{eq:fini_moderate_Pnrq_1}
\end{equation}
Moreover, applying the unilateral version of Gärtner-Ellis theorem, we get
\begin{align}
\frac{1}{r^{1-2\alpha}} \log \Prob\biggl( \sum_{i=q+1}^r (W_i (1 + \varepsilon r^{-\alpha}) - 1) \geqslant b r^{1-\alpha} \biggr)  \xrightarrow[r \to \infty]{} - J(a - \varepsilon) \label{eq:fini_moderate_Pnrq_2} .
\end{align}
So, applying the contraction principle to the continuous map $(y_1, \dots, y_q, s) \mapsto y_1 + \dots + y_q + s$ and
\[
\biggl(\frac{1}{r^{1 - \alpha}}\biggl( W_1 Z_{r,1}^{n-1} - 1, \dots, W_q Z_{r,q}^{n-1} - 1, \sum_{i=q+1}^r (W_i (1 + \varepsilon r^{-\alpha}) - 1) \biggr)\biggr)_{r \geqslant 2} ,
\]
and using \eqref{eq:fini_moderate_Pnrq_1} and \eqref{eq:fini_moderate_Pnrq_2}, one gets
\begin{align}
\limsup_{r \to \infty} \frac{1}{r^{1-2\alpha}} \log(P_{n, r, q})
 & \leqslant - J(a - \varepsilon) . \label{eq:fini_moderate_Pnrq_3}
\end{align}
%}

Applying the principle of the largest term to  \eqref{eq:moderate_fini_Pnrq0} and \eqref{eq:fini_moderate_Pnrq_3} and letting $\varepsilon \to 0$, one gets the desired upper bound:
\begin{align}
\limsup_{r \to \infty} \frac{1}{r^{1-2\alpha}} \log \Prob(r^\alpha (Z_r^n - 1) \geqslant a)
 \leqslant - J(a) .\label{eq:moderate_fini_majoration}
\end{align}
Finally, the large deviation principle stems from \eqref{eq:moderate_fini_minoration} and \eqref{eq:moderate_fini_majoration} and the fact that $J$ is decreasing on $\intervalleof{-\infty}{0}$ and increasing on $\intervallefo{0}{\infty}$.

\subsection{Proof of Theorem \ref{th:finite}, item \ref{th:finiteLD} (large deviations)} \label{sec:proof_finiteLD}

%New sketch of proof (to be written) :
%\begin{itemize}
%\item contraction version PGD;
%\item vérification que $I^n(a) = \Lambda_W^*(a)$ pour $a \leqslant 1$ (ou bien sur l'expression de $I^n$, ou via le PGD pour $a < 1$ et l'unicité de la fonction de taux).
%\end{itemize}
%
%
%\subsection{}
%
%
%
%
%
%\section{Old proof}
%
%Sketch of the proof:
%
%\begin{itemize}
%\item large deviations on the left: $\Prob(Z_r^n \leqslant a)$ for $a < 1$;
%\item large deviations on the right: $\Prob(Z_r^n \geqslant a)$ for $a > 1$;
%\item large deviation principle.
%\end{itemize}

The large deviations on the left for $(Z_r^{n})_{r\geqslant 2}$ are given in Proposition \ref{prop:leftLD}. Let us turn to the large deviations on the right. We proceed by induction. The case  $n = 1$ stems from Cramér's Theorem. In particular $I^1 = \Lambda_W^*$. Now let $n \in \N^*\setminus \{1\}$ and assume that $(Z_r^{n-1})_{r \geqslant 2}$ satisfies a large deviation principle at speed $(r)_{r \geqslant 2}$ with rate function $I^{n-1}$. Let $a > 1$ and $\varepsilon > 0$. Applying the contraction principle to the continuous map $(w, z, s) \mapsto w z + (1 - \varepsilon) s$ and $((W_1, Z_{r,1}^{n-1}, W_2+\dots+W_r)/r)_{r \geqslant 2}$, we get
\begin{align*}
\frac 1r \log
 &\, \Prob\left(Z_r^n \geqslant a \right) \\
 & \geqslant \frac 1r \log  \Prob\left(W_1 Z_{r,1}^{n-1} + \left(W_2+\dots+W_r\right)(1-\varepsilon)\geqslant ra,\ Z_{r,2}^{n-1},\dots, Z_{r,r}^{n-1} \geqslant 1-\varepsilon \right) \\
 & = \frac 1r \log \Prob\left(W_1 Z_{r,1}^{n-1} + \left(W_2+\dots+W_r\right)(1-\varepsilon)\geqslant ra\right)+ \frac{r-1}{r} \log \Prob\left( Z_r^{n-1} \geqslant 1-\varepsilon \right) \\
 & \xrightarrow[r \to \infty]{} - \inf \ensavec{cw + I^{n-1}(z) + \Lambda_W^*(s)}{w \geqslant 0,\ z \geqslant 0,\ s \geqslant 0,\  wz + (1-\varepsilon) s \geqslant a} \\
 & \geqslant - \inf \ensavec{cw + I^{n-1}(z) + \Lambda_W^*(s)}{w \geqslant 0,\ z\in \intervalleff{1}{a},\ s\in \intervalleff{1}{a},\  wz + (1-\varepsilon) s = a} \\
 %& = - \inf \ensavec{\frac{c(a - s(1 - \varepsilon))}{z} + I^{n-1}(z) + \Lambda_W^*(s)}{z\in \intervalleff{1}{a},\ s\in \intervalleff{1}{a},\ a - s(1 - \varepsilon) \geqslant 0} \\
 & = - \inf \ensavec{\frac{c(a - s(1 - \varepsilon))}{z} + I^{n-1}(z) + \Lambda_W^*(s)}{z\in \intervalleff{1}{a},\ s\in \intervalleff{1}{a}} .
\end{align*}
Hence,
\begin{align*}
\liminf_{r \to \infty} \frac 1r \log
 &\, \Prob\left(Z_r^n \geqslant a \right) \\
 & \geqslant - \inf_{\varepsilon > 0} \inf_{\substack{1 \leqslant z \leqslant a \\ 1 \leqslant s \leqslant a}} \biggl[\frac{c(a - s(1 - \varepsilon))}{z} + I^{n-1}(z) + \Lambda_W^*(s) \biggr] \\
 & = - \inf_{\substack{1 \leqslant z \leqslant a \\ 1 \leqslant s \leqslant a}} \biggl[\frac{c(a - s)}{z} + I^{n-1}(z) + \Lambda_W^*(s) + \inf_{\varepsilon > 0} \frac{c s \varepsilon}{z} \biggr] \\
 & = - \inf_{\substack{1 \leqslant z \leqslant a \\ 1 \leqslant s \leqslant a}} \biggl[\frac{c(a - s)}{z} + I^{n-1}(z) + \Lambda_W^*(s) \biggr] \\
 & = - \inf \ensavec{cw + I^{n-1}(z) + \Lambda_W^*(s)}{w \geqslant 0,\ z\in \intervalleff{1}{a},\ s \in \intervalleff{1}{a},\ wz + s = a} \\
 & = - I^n(a) ,
\end{align*}
by \eqref{eq:Inrecurrence2}.

\medskip

Now let us establish the upper bound. Let $\varepsilon > 0$. For all $r\in \N^* \setminus \{ 1 \}$ and for $q\in \intervallentff{0}{r}$, we define
\[
P_{n,r,q} \defeq \Prob\left(Z_r^n \geqslant a,\ Z_{r,1}^{n-1},\dots,Z_{r,q}^{n-1}\geqslant 1+\varepsilon,\ Z_{r,q+1}^{n-1},\dots, Z_{r,r}^{n-1} < 1+\varepsilon \right) ,
\]
so that 
\[
\Prob\left( Z_r^n \geqslant a \right)= \sum_{q=0}^r \binom{r}{q} P_{n,r,q} .
\]
For $q_0 \in \intervallentff{0}{r}$ and $\varepsilon' \in \intervalleoo{0}{\varepsilon}$, one has by induction
%\begin{align*}
%\sum_{q=q_0}^r \binom{r}{q} P_{n,r,q}
% &\leqslant \sum_{q=q_0}^r \binom{r}{q} \Prob(Z_r^{n-1} \geqslant 1+\varepsilon)^q \\
% &\leqslant \sum_{q=q_0}^r \left(re^{-r\Lambda_W^*(1+\varepsilon)}\right)^q \\
% &\leqslant \frac{\left(re^{-r\Lambda_W^*(1+\varepsilon)}\right)^{q_0}}{1-re^{-r\Lambda_W^*(1+\varepsilon)}} \\
% &\leqslant e^{-r I^n(a)}
%\end{align*}
\begin{align}
\sum_{q=q_0}^r \binom{r}{q} P_{n,r,q}
 &\leqslant \sum_{q=q_0}^r \binom{r}{q} \Prob(Z_r^{n-1} \geqslant 1+\varepsilon)^q \notag\\
 &\leqslant \sum_{q=q_0}^r \left(re^{-rI^{n-1}(1+\varepsilon')}\right)^q \notag\\
 &\leqslant \frac{\left(re^{-rI^{n-1}(1+\varepsilon')}\right)^{q_0}}{1-re^{-rI^{n-1}(1+\varepsilon')}} \notag\\
 &\leqslant e^{-r I^n(a)} ,\label{eq:large_fini_Pnrq0}
\end{align}
as soon as $q_0 I^{n-1}(1+\varepsilon') > I^n(a)$ (we use Proposition \ref{prop:Indecroissante1}, item \ref{prop:Indecroissante_positive}) and $r$ is large enough. Now,
\begin{align}
\lim_{r\to\infty} \frac 1r \log P_{n,r,0} \leqslant \lim_{r\to\infty} \frac 1r \log \Prob\Bigl(\sum_{i=1}^r W_i\geqslant \frac{ra}{1+\varepsilon}\Bigr)=  -\Lambda_W^* \Bigl(\frac{a}{1+\varepsilon}\Bigr) \leqslant - I^n\Bigl(\frac{a}{1+\varepsilon}\Bigr) .\label{eq:large_fini_Pnr0}
\end{align}
Finally, let $q\in \intervallentff{1}{q_0-1}$.  Applying the contraction principle to the continuous map $(y_1, \dots, y_q, s) \mapsto y_1 + \dots + y_q + s(1 + \varepsilon)$ and $((W_1 Z_{r,1}^{n-1}, \dots, W_q Z_{r,q}^{n-1}, W_{q+1}+\dots+W_r)/r)_{r \geqslant 2}$, we get
\begin{align*}
 \frac 1r \log P_{n,r,q}
 & \leqslant \frac 1r \log \Prob\left(W_1 Z_{r,1}^{n-1}+\dots+ W_q Z_{r,q}^{n-1}+\left(W_{q+1}+\dots+W_r\right)(1+\varepsilon) \geqslant ra, \right. \\
 & \hspace{4cm} \left. Z_{r,1}^{n-1},\dots,Z_{r,q}^{n-1}\geqslant 1+\varepsilon,\ Z_{r,q+1}^{n-1},\dots,Z_{r,r}^{n-1}< 1+\varepsilon \right) \\
 & \leqslant \frac 1r  \log \Prob\left(W_1 Z_{r,1}^{n-1}+\dots+ W_q Z_{r,q}^{n-1}+\left(W_{q+1}+\dots+W_r\right)(1+\varepsilon) \geqslant ra \right) \\
 & \xrightarrow[r\to \infty]{} -\inf \{g^{n-1}(y_1)+\dots+g^{n-1}(y_q)+\Lambda_W^*(s)\ ; \\
 & \hspace{3cm} y_1,\ \dots,\ y_q \geqslant 0,\ s \geqslant 0,\ y_1+\dots+y_q+(1+\varepsilon )s\geqslant a \} \\
% & = -\inf \ensavec{g^{n-1}(y)+\Lambda_W^*\left(\frac{a-y}{1+\varepsilon}\right)}{y \geqslant 0} ,
 & = -\inf \ensavec{g^{n-1}(y)+\Lambda_W^*(s)}{y \geqslant 0,\ s \geqslant 0,\ y + (1+\varepsilon)s \geqslant a}\\
  & \eqdef - B_\varepsilon ,
\end{align*}
where $g^{n-1}$ is the rate function of $(W Z_r^{n-1} / r)_{r \geqslant 2}$ at speed $(r)_{r \geqslant 2}$, that is the concave function defined by
\begin{align*}
g^{n-1}(y)
 & \defeq \inf \ensavec{cw + I^{n-1}(z)}{w \geqslant 0,\ z > 0,\ wz = y}  = \inf \ensavec{\frac {cy}{z} + I^{n-1}(z)}{z > 0} .
\end{align*}
Following the same lines as in the proof of Proposition \ref{prop:Indecroissante1}, item \ref{prop:Indecroissante_recurrence2}, we obtain that, for $\varepsilon \leqslant a - 1$,
\[
B_\varepsilon =  \inf \ensavec{cw + I^{n-1}(z) + \Lambda_W^*(s)}{w \geqslant 0,\ z \in \intervalleff{1}{a},\ s \in \intervalleff{1}{a/(1+\varepsilon)},\ wz + (1 + \varepsilon)s = a}.
\]
Hence,
\begin{align*}%\label{eq:large_fini_B_epsilon}
\limsup_{r \to \infty}
 \frac{1}{r} \log P_{n, r, q} 
\leqslant - \sup_{0 < \varepsilon \leqslant a-1}  B_\varepsilon.
% & = - \sup_{0 < \varepsilon < a-1} \inf \ensavec{cw + I^{n-1}(z) + \Lambda_W^*(s)}{w \geqslant 0,\ z \geqslant 1,\ s \geqslant 1,\ wz + s(1 + \varepsilon) = a} . 
\end{align*}

If $c=\infty$, then, since $\Lambda_W^*$ is left continuous at $a$,
\begin{align*}
\limsup_{r \to \infty} \frac{1}{r} \log P_{n, r, q}
 \leqslant - \sup_{\varepsilon > 0} \Lambda_W^*\Bigl(\frac{a}{1+\varepsilon}\Bigr)
 = - \Lambda_W^*(a)
 = - I^n(a).
\end{align*}
Assume now that $c < \infty$. Then,
\begin{align}
\limsup_{r \to \infty}
 \frac{1}{r} \log P_{n, r, q} & \leqslant - \sup_{0<\varepsilon \leqslant a-1} \inf_{\substack{1 \leqslant z \leqslant a,\\ 1\leqslant s\leqslant \frac{a}{1+\varepsilon}}}\Bigl[ \frac {c(a-s)}{z} + I^{n-1}(z)+\Lambda_W^*(s)-\frac{c\varepsilon s}{z}\Bigr] \notag\\
% & \leqslant - \sup_{0<\varepsilon \leqslant a-1} \inf_{\substack{1 \leqslant z \leqslant a,\\ 1\leqslant s\leqslant \frac{a}{1+\varepsilon}}}\Bigl[ \frac {c(a-s)}{z} + I^{n-1}(z)+\Lambda_W^*(s)-\frac{c\varepsilon a}{1+\varepsilon}\Bigr] \notag\\
 & \leqslant - \sup_{0<\varepsilon \leqslant a-1} \inf_{\substack{1 \leqslant z \leqslant a\\1 \leqslant s \leqslant a}}\Bigl[ \frac {c(a-s)}{z} + I^{n-1}(z)+\Lambda_W^*(s)-c\varepsilon a\Bigr]
 %\quad \text{(since $z\geqslant 1$ and $s \leqslant a$)}
 \notag\\
 & = - \inf_{\substack{1 \leqslant z \leqslant a\\1 \leqslant s \leqslant a}}\Bigl[ \frac {c(a-s)}{z} + I^{n-1}(z)+\Lambda_W^*(s)\Bigr] \notag\\
 %& = - \inf_{\substack{w \geqslant 0,\\ 1\leqslant z\leqslant a}}\Bigl[ c w + I^{n-1}(z)+\Lambda_W^*(a - wz)\Bigr]. \notag
 & = - \inf \ensavec{c w + I^{n-1}(z)+\Lambda_W^*(s)}{w \geqslant 0,\ z \in \intervalleff{1}{a},\ s \in \intervalleff{1}{a},\ wz + s = a} \notag \\
 & = - I^n(a) \label{eq:large_fini_Pnrq}
\end{align}
by \eqref{eq:Inrecurrence2}. We conclude by applying the principle of the largest term to \eqref{eq:large_fini_Pnrq0}, \eqref{eq:large_fini_Pnr0}, and \eqref{eq:large_fini_Pnrq}.

%%%%%%%%%%%%%%%%%%%%%%
%\newpage
%
%\begin{prop}
%For $ y\geqslant 0$, one has
%\[
%\lim_{n\to+\infty} \frac 1n \log \Prob\left(X_1\Snbar{1} \geqslant ny\right) = -\left(\sqrt{1+4y}-\log  (1 + \sqrt{1+4y}) + \log 2 - 1 \right)\eqdef -q(y)\, .
%\]
%\end{prop}
%
%Notice that the function $q$ is concave.
%
%\begin{proof}
%Applying the contraction principle,
%\begin{align*}
%\frac 1n \log \Prob\left(X_1\Snbar{1} \geqslant ny\right) & = \frac 1n \log \Prob\left(X_1 \Sn{1} \geqslant n^2y\right)\\
%&\longrightarrow -\inf_{s>0} \left[\frac ys +\Lambda_W^*(s)\right]\\
%&= -\inf_{s>0} \left[\frac ys +s-\log s-1\right].
%\end{align*}
%The infimum is attained at $s=(1 + \sqrt{1+4y})/2$ from which the result follows.
%\end{proof}

%\section{Generalization to $Z_r^n$}
%
%Cf. notation article Alain (2004).
%
%
%\begin{thm}
%For $ a\geqslant 1$, for $n \in \N^*$, one has
%\[
%\lim_{r\to+\infty} \frac 1r \log \Prob\left(Z_r^n \geqslant a\right) = -\inf_{0\leqslant y\leqslant a} \bigl[ q_n(y)+\Lambda_W^*(a-y) \bigr] \eqdef -I^n(a)\, ,
%\]
%where $q_n$ is defined by induction by
%\[
%q_n(y)\defeq \inf_{s>0} \Bigl[ \frac ys+I^{n-1}(s) \Bigr].
%\]
%
%
%\end{thm}

\subsection{Proof of Theorem \ref{th:finite}, item \ref{th:finiteVL} (very large deviations)}

\emph{Left deviations} ---
%By the law of large numbers, for $a>0$, $\Prob(Z_r^n \leqslant r^\alpha a) \to 1$ as $r \to \infty$ and,
For $a<0$, $\Prob(Z_r^n\leqslant r^\alpha a) = 0$. Assume now that $a=0$. Let $p \defeq \Prob(W=0)$ and, for all $r \in \N^* \setminus \{ 1 \}$ and $n \in \N^*$, let $q_r^n \defeq \Prob(Z_r^n \leqslant 0) = \Prob(Z_r^n = 0)$. Obviously, $q_r^1 = p^r$ and, using \eqref{eq:E_finite},
\begin{align*}
q_r^n
 & =\Prob(Z_r^n = 0)
 = \Prob(W Z_r^{n-1} = 0)^r
 = (1 - \Prob(W \neq 0) \Prob(Z_r^{n-1} \neq 0))^r
 = (p + (1-p) q_r^{n-1})^r .
\end{align*}
If $p = 0$, then $q_r^n = 0$ for all $r \in \N^* \setminus \{ 1 \}$ and $n \in \N^*$. Otherwise, $0 < p < 1$ (recall that $\Espe[W]=1$ implies $p < 1$), $r^{-1} \log(q_r^1) = \log(p)$ and one proves by induction that
\[
\frac{1}{r} \log(q_r^n)
 = \log(p + (1-p) q_r^{n-1})
 \xrightarrow[r \to \infty]{} \log(p) .
\]
So,
\begin{align*}
- \frac{1}{r^{1+\alpha/n}} \log \Prob(Z_r^n \leqslant r^\alpha a)
 \xrightarrow[r \to \infty]{} \begin{cases}
\infty & \text{if $a < 0$ or ($a=0$ and $p = 0$)} \\
0 & \text{if ($a = 0$ and $p > 0$).} % or $a>0$
\end{cases}
\end{align*}
% Remark:
% \begin{align*}
% -\frac 1r \log \Prob(Z_r^n\leqslant r^\alpha a)
%  \xrightarrow[r \to \infty]{} \begin{cases}
% 0 & \text{for $a>0$}\\
% -\log p & \text{for $a=0$}\\
% \infty & \text{for $a<0$.}
% \end{cases}
% \end{align*}

\emph{Right deviations} ---  The case $a = 0$ is obvious. Let $a > 0$. Let us prove the minoration. For all $k \in \N^*$, we introduce the notation $1^k$ to represent the word $1, 1, \dots, 1$ of length $k$. One has
\begin{align*}
\frac{1}{r^{1+\alpha/n}}\log \Prob(Z_r^n\geqslant r^\alpha a)
 & \geqslant \frac{1}{r^{1+\alpha/n}}\log \Prob\Bigl(\frac{W_{1^1} \cdots W_{1^n}}{r^n}\geqslant r^\alpha a\Bigr)\\
 & \geqslant \frac{n}{r^{1+\alpha/n}} \log \Prob\Bigl(W\geqslant r( r^\alpha a)^{1/n}\Bigr)\\
 & \xrightarrow[r \to \infty]{} - c n  a^{1/n} .%\label{eq:very_fini_minoration}
\end{align*}

As for the majoration, we proceed by induction. The case $n=1$ corresponds to the very large deviations for i.i.d.\ random variables. Since \eqref{eq:tail_W} is satisfied, we may apply \cite[Proposition 1.1 and Remark 1.1]{broniatowski1994extended} to obtain 
\[
\log \Prob(Z_r^1\geqslant r^\alpha a) \sim - r \Lambda_W^*(r^\alpha a)\sim -c r ^{1+\alpha} a,
\] 
as $r \to \infty$ where the last estimate stems from Proposition \ref{prop:Lambda}, item \ref{prop:Lambda_star}. Then we follow the same lines as in the upper bound for the large and moderate deviations. Let $\varepsilon > 0$.  For all $r\in \N^* \setminus \{ 1 \}$ and for $q\in \intervallentff{0}{r}$, we define
\[
P_{n,r,q} \defeq \Prob\left(Z_r^n \geqslant r^\alpha a,\ Z_{r,1}^{n-1},\dots,Z_{r,q}^{n-1}\geqslant r^{\alpha(n-1)/n} \varepsilon,\ Z_{r,q+1}^{n-1},\dots, Z_{r,r}^{n-1} < r^{\alpha(n-1)/n} \varepsilon \right) ,
\]
so that 
\[
\Prob\left( Z_r^n \geqslant r^\alpha a \right)= \sum_{q=0}^r \binom{r}{q} P_{n,r,q} .
\]
For $q_0 \in \intervallentff{0}{r}$ and $c' \in \intervalleoo{0}{c}$, one has by induction
\begin{align*}
\sum_{q=q_0}^r \binom{r}{q} P_{n,r,q}
 & \leqslant \sum_{q=q_0}^r \binom{r}{q} \Prob(Z_r^{n-1} \geqslant r^{\alpha(n-1)/n} \varepsilon)^q  \\
 & \leqslant \sum_{q=q_0}^r \bigl(re^{-r^{1+\alpha/n}c' (n-1)  \varepsilon^{1/(n-1)}}\bigr)^q  \\
 & \leqslant \frac{\bigl( re^{-r^{1+\alpha/n} c' (n-1) \varepsilon^{1/(n-1)}}\bigr)^{q_0}}{1 - re^{-r^{1+\alpha/n} c'(n-1)  \varepsilon^{1/(n-1)}}} \\
 & \leqslant e^{-r^{1+\alpha/n} c' n  a^{1/n}} ,
\end{align*}
as soon as $(n-1) q_0 \varepsilon^{1/(n-1)} > n a^{1/n}$ and $r$ is large enough. Hence 
\begin{align}\label{eq:very_fini_Pnrq0} 
\limsup_{r\to \infty} \frac{1}{r^{1+\alpha/n}}\log \sum_{q=q_0}^r \binom{r}{q} P_{n,r,q}\leqslant -c n  a^{1/n}.
\end{align}
Now,
\begin{align*}
P_{n,r,0}
 = \Prob\left(\sum_{i=1}^r W_iZ_{r,i}^{n-1} \geqslant r^{1+\alpha} a,\ Z_{r,1}^{n-1},\dots, Z_{r,r}^{n-1} < r^{\alpha(n-1)/n} \varepsilon \right)
 \leqslant \Prob\left(\sum_{i=1}^r W_i \geqslant r^{1+\alpha/n} \frac{a}{\varepsilon} \right).
\end{align*}
Using the exponential Chebyshev inequality, for all $t \in \intervalleoo{0}{c}$, one has
\begin{align}
\limsup_{r \to \infty} \frac{1}{r^{1+\alpha/n}} \log \Prob\left(\sum_{i=1}^r W_i \geqslant r^{1+\alpha/n} \frac{a}{\varepsilon} \right)
 & \leqslant \limsup_{r \to \infty} \frac{1}{r^{1+\alpha/n}} (- t r^{1+\alpha/n} a / \varepsilon + r \Lambda_W(t)) \notag \\
 %\leqslant  - \frac{ta}{\varepsilon} + \limsup_{r \to \infty} \frac{1}{r^{\alpha/n}} \Lambda_W( t )\label{eq:very_fini_Pnr0}
 & = - \frac{ta}{\varepsilon}. \label{eq:very_fini_Pnr0}
\end{align}
Hence, for $\varepsilon > 0$ small enough, one gets
\[
\limsup_{r \to \infty} \frac{1}{r^{1+\alpha/n}} \log P_{n,r,0}
 \leqslant - c n a^{1/n}.
\]

Finally, let $q\in \intervallentff{1}{q_0-1}$. One has
\begin{align*}
P_{n,r,q}
 & \leqslant \Prob\left(\sum_{i=1}^q W_iZ_{r,i}^{n-1}+ \sum_{i=q+1}^r W_i r^{\alpha(n-1)/n} \varepsilon \geqslant r^{1+\alpha} a \right) .
\end{align*}
We want to apply the contraction principle to the continuous map $(y_1, \dots, y_q, s) \mapsto y_1 + \dots + y_q + s$ and
\[
\biggl(\frac{1}{r^{1+\alpha}}\biggl( W_1 Z_{r, 1}^{n-1}, \dots, W_q Z_{r, q}^{n-1}, \sum_{i=q+1}^r W_i r^{\alpha(n-1)/n} \varepsilon \biggr) \biggr)_{r \geqslant 2} .
\]
As for the first $q$ variables, applying the contraction principle to the continuous map $(w, z) \mapsto w z$ and $(W / r^{1 + \alpha / n}, Z_r^{n-1} / r^{\alpha (n-1) / n})_{r \geqslant 2}$, we get by induction
\begin{align*}
\frac{1}{r^{1+\alpha/n}}\log \Prob( W Z_{r}^{n-1} \geqslant r^{1+\alpha} y)
 & \xrightarrow[r \to \infty]{} -\inf \ensavec{c w + (n-1) c z^{1/(n-1)}}{w \geqslant 0,\ z > 0,\ wz = y} \\
 & = -\inf \ensavec{\frac{cy}{z} + (n-1) c z^{1/(n-1)}}{z > 0} \\
 & = - c n y^{1/n}
\end{align*}
(the infimum is attained at $z=y^{(n-1)/n}$). Hence, applying the contraction principle and \eqref{eq:very_fini_Pnr0}, we get
\begin{align}
\frac{1}{r^{1+\alpha/n}} \log P_{n,r,q}
 & \leqslant -c \inf\Bigl\{  n(y_1^{1/n}+\dots+y_q^{1/n})+\frac{s}{\varepsilon};\, y_1+\dots+y_q+s=a \Bigr\} 
  = -cn a^{1/n} \label{eq:very_fini_Pnrq} .
\end{align}
Applying the principle of the largest term to \eqref{eq:very_fini_Pnrq0}, \eqref{eq:very_fini_Pnr0}, and \eqref{eq:very_fini_Pnrq}, one gets
\begin{align*}
\limsup_{r \to \infty} \frac{1}{r^{1+\alpha/n}} \log \Prob(Z_r^n \geqslant r^{\alpha} a)
 \leqslant - cn a^{1/n} .%\label{eq:very_fini_majoration}
\end{align*}

\section{Infinite tree}\label{sec:proofinfinite}

The heart of the proofs in the finite tree is the recursion formula \eqref{eq:E_finite} for $Z_r^n$ which allows to proceed by induction. However, in the infinite tree, the analogue of \eqref{eq:E_finite} is the fixed point equation \eqref{eq:E_infinite} and there is no more recursion. From now on, we set $Z_r \defeq Z_r^\infty$ for all $r\in \N^* \setminus \{ 1 \}$. By \cite[Théorèmes 1 et 3]{kahane1976}, if $\esssup(W) = \infty$ and $r > \exp \Espe[W \log(W)]$, then
\[
\forall t > 0 \quad
\Espe[e^{t Z_r}] = \infty .
\]
Therefore, under assumption \eqref{eq:tail_W} and $\esssup(W) = \infty$ (see Table \ref{table:summary}), the standard exponential Markov inequality cannot be used to prove the upper bounds. Here, we bypass this problem by bounding the moments of $Z_r$. For the moderate and very large deviations, we optimize the bound given by Markov inequality over the moments to derive the exact upper bound. As for the large deviations, we prove that (see Corollary \ref{cor:moments}, item \ref{cor:moments_grands})
\[
\limsup_{r \to \infty} \frac{1}{r} \log \Espe[Z_r^{\eta r}] = O(\eta^2)
\]
which is the key argument to prove that the rate function is non degenerate and is indeed the limit of the rate functions in the finite trees (see Theorem \ref{th:finite}, item \ref{th:finiteLD}).

\subsection{Upper bounds for the moments of \texorpdfstring{$Z_r$}{Zr}} \label{sec:moments_Zr}

Let us turn to the control of the moments of $Z_r$. Using the fact that positive martingales converge almost surely, we may let $n \to \infty$ in \eqref{eq:E_finite} to get
\[
Z_r = \frac 1r \sum_{i=1}^r W_i Z_{r, i} \quad \text{a.s.}
\]
where the random variables $Z_{r, 1}$, ..., $Z_{r, r}$, $W_1$, ..., $W_r$ are independent, the $W_i$ (resp.\ $Z_{r, i}$) having the same distribution as $W$ (resp.\ $Z_r$). It is proved in \cite[Théorème 1]{kahane1976} that, for
\begin{equation} \label{eq:r_grand}
r > \exp \Espe[W \log(W)],
\end{equation}
$Z_r$ is the unique solution $Z$ of \eqref{eq:E_infinite} such that $\Espe[Z] = 1$. Moreover, under \eqref{eq:r_grand}, by \cite[Théorème 2]{kahane1976}, $\Espe[Z_r^h] < \infty$ if and only if $\Espe[W^h] < r^{h-1}$, which is equivalent to $h < \chi(r)$ for
\[
\chi(r) \defeq \sup \ensavec{h \in \intervallefo{1}{\infty}}{\Espe[W^h] < r^{h-1}} .
\]

% OLD Proposition
% \begin{prop}\label{prop:momk}
% For all $\varepsilon \in \intervalleoo{0}{c}$, there exists $h_0 \in \N$ such that
% \begin{align*}%\label{eq:momk}
% \forall h \geqslant h_0 \quad
% \frac{1}{(c+\varepsilon)^h} \leqslant
% \frac{\Espe[W^h]}{h!} \leqslant \frac{1}{(c-\varepsilon)^h} .
% \end{align*}
% \end{prop}

In the sequel, we always consider values of $r$ which satisfy \eqref{eq:r_grand}. Moreover, for all $(a, b) \in \R^2$, we denote by $\intervallentff{a}{b}$ the set of integers between $a$ and $b$, i.e.\ $\enstq{h \in \Z}{a \leqslant h \leqslant b}$. Similarly, with obvious definition, we may also use the variants $\intervallentfo{a}{b}$, etc.

\begin{prop}[Moments of $W$] \label{prop:moments_W}
Assume \eqref{eq:tail_W}.
\begin{enumerate}
\item \label{prop:moments_W_momk} One has
\[
\frac{1}{h} \log \frac{\Espe[W^h]}{h!} \xrightarrow[h \to \infty]{} - \log(c).
\]

\item \label{prop:moments_W_chi} One has
\begin{equation*} %\label{eq:chi_limit}
\frac{\chi(r)}{r} \xrightarrow[r \to \infty]{} ce .
\end{equation*}
%In particular
%\[
%\forall \eta \in \intervalleoo{0}{ce} \quad \exists r_0 \geqslant 2 \quad \forall r \geqslant r_0 \quad \eta r < \chi(r)
%\]

\item \label{prop:moments_W_terme_inverse} For all $\eta \in \intervalleoo{0}{ce}$, for all $r$ large enough,
\[
\sup_{h \in \intervallentff{2}{\eta r}} \frac{\Espe[W^h]}{r^{h-1}}
 = \frac{\Espe[W^2]}{r}.
\]
\end{enumerate}
\end{prop}

\begin{proof}[Proof of Proposition \ref{prop:moments_W}] \leavevmode
\begin{enumerate}
\item Let $0 < \varepsilon' < \varepsilon < c$ and let $A$ be such that, for all $w\geqslant A$,
\[
\Prob(W > w)\leqslant e^{-(c-\varepsilon') w}.
\]
Then, integrating by parts,
\begin{align*}
\Espe[W^h]
% & = \int_0^{\infty} \Prob(W^h > t) dt \\
% & \leqslant A + \int_A^\infty \Prob(W^h > t) dt \\ 
 & \leqslant A + \int_A^\infty e^{-(c-\varepsilon') t^{1/h}} dt
 = A + h \int_{A^{1/h}}^\infty e^{-(c-\varepsilon') w} w^{h-1} dw \\
% & = A - h \left[\frac{w^{h-1} e^{-(c-\varepsilon') w}}{c-\varepsilon'} \right]_{A^{1/h}}^\infty + \frac{h(h-1)}{c-\varepsilon'} \int_{A^{1/h}}^\infty e^{-(c-\varepsilon') w} w^{h-2} dw
 & = A + \frac{h A^{(h-1)/h}}{c-\varepsilon'} e^{-(c-\varepsilon') A^{1/h}} + \frac{h(h-1)}{c-\varepsilon'} \int_{A^{1/h}}^\infty e^{-(c-\varepsilon') w} w^{h-2} dw.
\end{align*}
By induction, we get
\begin{align*}
\Espe[W^h]
 & \leqslant A + \frac{h!}{(c-\varepsilon')^h} e^{-(c-\varepsilon') A^{1/h}} \sum_{l=0}^{h-1} \frac{((c-\varepsilon') A^{1/h})^l}{l!}
 \leqslant A + \frac{h!}{(c-\varepsilon')^h}
% & = \frac{h!}{(c-\varepsilon)^h} \biggl( \frac{A (c - \varepsilon)^h}{h!} + \frac{(c - \varepsilon)^h}{(c - \varepsilon')^h} \biggr)
 = o \biggl( \frac{h!}{(c-\varepsilon)^h} \biggr)
\end{align*}
as $h \to \infty$, whence the result.
Similarly, the lower bound stems from
\begin{align*}
\Espe[W^h] & \geqslant \int_A^\infty e^{-(c+\varepsilon) t^{1/h}}dt.
\end{align*}

\item Let $\eta > 0$. It suffices to prove that, for all $r$ large enough,
\[
\frac{\Espe[W^{\floor{\eta r}}]}{r^{\floor{\eta r} - 1}} \begin{cases}
< 1 & \text{for $\eta < ce$} \\
> 1 & \text{for $\eta > ce$}.
\end{cases}
\]
If $\eta < ce$, let $\varepsilon > 0$ be such that $\eta < (c - \varepsilon) e$. Using item \ref{prop:moments_W_momk} and Stirling bounds (see \cite[Section 3]{Artin_1964_Gamma})
\begin{equation}\label{eq:stirling}
\forall h \in \intervalleoo{0}{\infty} \quad
\left(\frac{h}{e}\right)^h \sqrt{2\pi h}\leqslant \Gamma(h+1) \leqslant \left(\frac{h}{e}\right)^h \sqrt{2\pi h} e^{\frac{1}{12h}},
\end{equation}
we get, for $r$ large enough,
\begin{equation} \label{eq:mom_eta_r}
\frac{\Espe[W^{\floor{\eta r}}]}{r^{\floor{\eta r}-1}}
 \leqslant \frac{\floor{\eta r}!}{r^{\floor{\eta r}-1}(c - \varepsilon)^{\floor{\eta r}}}
 \leqslant \biggl( \frac{\floor{\eta r}}{(c - \varepsilon) e r} \biggr)^{\floor{\eta r}} r \sqrt{2 \pi \floor{\eta r}}\ e^{\frac{1}{12 \floor{\eta r}}}
 < 1.
\end{equation}
Same argument for $\eta > ce$.

\item Let $\eta \in \intervalleoo{0}{ce}$. Since the function $h \mapsto \Espe[W^h] / r^{h-1}$ is log-convex (see \cite[p.\ 132]{kahane1976}), if $\eta r \geqslant 2$, then
\[
\sup_{h \in \intervallentff{2}{\eta r}} \frac{\Espe[W^h]}{r^{h-1}}
 = \max \biggl\{ \frac{\Espe[W^2]}{r}, \frac{\Espe[W^{\floor{\eta r}}]}{r^{\floor{\eta r}-1}} \biggr\}.
\]
Let $\varepsilon > 0$ be such that $\eta < (c - \varepsilon) e$. The conclusion stems from the fact that \eqref{eq:mom_eta_r} yields besides
\[
\frac{\Espe[W^{\floor{\eta r}}]}{r^{\floor{\eta r}-1}}
 = o\Bigl( \frac{1}{r} \Bigr). \qedhere
\]
\end{enumerate}
\end{proof}

Now, it is proved in \cite[Equation (13)]{kahane1976} that, for all $r > \exp \Espe[W \log(W)]$, we have the following recursion formula:
\begin{equation} \label{eq:kahane_peyriere}
\forall h \in \intervallentfo{2}{\chi(r)} \quad
\Espe[Z_r^h]
 = \frac{h!}{r^h \bigl( 1 - \frac{\Espe[W^h]}{r^{h-1}}\bigr)}\sum_{\substack{h_1+\dots+h_r=h,\\ 0 \leqslant h_j\leqslant h-1}} \prod_{j=1}^r \frac{\Espe[W^{h_j}] \Espe[Z_r^{h_j}]}{h_j!} .
\end{equation}

Let us turn to bounds on the moments of $Z_r$. First, we consider the small moments.

\begin{prop} \label{prop:Zrh_rec}
For all $\delta > 0$, there exists $\eta \in \intervalleoo{0}{1 \wedge ce}$ such that, for all $r$ large enough,
\[
\forall h \in \intervallentff{1}{\eta r} \quad \Espe[Z_r^h] \leqslant \exp\biggl\{ \frac{h^2}{2 (r - h)} \biggl( \Var(W) + \delta + \frac{C}{h} \biggr) \biggr\} ,
\]
with $C = 2 \Espe[W^2] + 13/12$.
\end{prop}

\begin{proof}[Proof of Proposition \ref{prop:Zrh_rec}]
Let $\delta > 0$. Let us fix some $\eta \in \intervalleoo{0}{1 \wedge ce}$. By Proposition \ref{prop:moments_W}, item \ref{prop:moments_W_terme_inverse}, there exists $r_0 > \exp \Espe[W \log(W)]$ such that
\begin{equation} \label{eq:majoration_momemnt_W}
\forall r \geqslant r_0 \quad
\sup_{h \in \intervallentff{2}{\eta r}} \frac{\Espe[W^h]}{r^{h-1}}
 = \frac{\Espe[W^2]}{r}
 \leqslant \frac{1}{2}.
\end{equation}
In particular, for all $r \geqslant r_0$, $\eta r < \chi(r)$. Now, let us fix $r \geqslant r_0$. For all $h \geqslant 1$, denote by $E(h)$ the statement:
\[
\Espe[Z_r^h] \leqslant \exp\biggl\{ \frac{h^2}{2 (r - h)} \biggl( \Var(W) + \delta + \frac{C}{h} \biggr) \biggr\}
\]
and let us prove $E(h)$ by induction for $h \in \intervallentff{1}{\eta r}$. Since $\Espe[Z_r] = 1$, $E(1)$ is obvious. Now, in case $\eta r \geqslant 2$, let $h \in \intervallentff{2}{\eta r}$ and assume that $E(k)$ is satisfied for all $k \in \intervallentff{1}{h-1}$. From \eqref{eq:kahane_peyriere} and the fact that $h \leqslant \eta r < \chi(r)$, one has
\begin{align}
\Espe[Z_r^h]
 & = \frac{h!}{r^h \bigl( 1 - \frac{\Espe[W^h]}{r^{h-1}}\bigr)} \sum_{\substack{m_0+\dots+m_{h-1}=r,\\ m_1+\dots+(h-1)m_{h-1}=h}} \binom{r}{m_0,\dots,m_{h-1}} \prod_{k=0}^{h-1} \biggl( \frac{\Espe[W^k] \Espe[Z_r^k]}{k!} \biggr)^{m_k} \notag \\
 & = \frac{h! r!}{r^h \bigl( 1 - \frac{\Espe[W^h]}{r^{h-1}}\bigr)} \sum_{\substack{m_0+\dots+m_{h-1}=r,\\ m_1+\dots+(h-1)m_{h-1}=h}} \prod_{k=0}^{h-1} \frac{1}{m_k!} \biggl( \frac{\Espe[W^k] \Espe[Z_r^k]}{k!} \biggr)^{m_k} , \label{eq:Zrh}
\end{align}
where $m_k$ stands for the number of $h_j$ equal to $k$ (multiplicity of $k$ in the composition $(h_1,\dots,h_r)$ of $h$). Then, the latter sum above is equal to
\begin{align}
 & \sum_{\substack{m_0+\dots+m_{h-1}=r,\\ m_1+\dots+(h-1)m_{h-1}=h}} \hspace{-0.7cm} \frac{(r-h)^{m_0-(r-h)}}{m_0!} \cdot \frac{h^{-(h-m_1)}}{m_1!} \prod_{k=2}^{h-1} \frac{1}{m_k!}\left(\frac{\Espe[W^k] \Espe[Z_r^k] h^k}{k! (r-h)^{k-1}}\right)^{m_k} \notag \\
 & \leqslant \sum_{\substack{m_0+\dots+m_{h-1}=r,\\ m_1+\dots+(h-1)m_{h-1}=h}} \frac{1}{(r-h)!} \cdot \frac{1}{h!} \prod_{k=2}^{h-1} \frac{1}{m_k!}\left(\frac{\Espe[W^k] \Espe[Z_r^k] h^k}{k! (r-h)^{k-1}}\right)^{m_k} \notag \\
 %& = \frac{1}{(r-h)! h!} \sum_{\substack{m_0+\dots+m_{h-1}=r,\\
%m_1+\dots+(h-1)m_{h-1}=h}}  \prod_{k=2}^{h-1} \frac{1}{m_k!}\left(\frac{\Espe[W^k] \Espe[Z_r^k] h^k}{k! (r-h)^{k-1}}\right)^{m_k} \notag \\
 & \leqslant \frac{1}{(r-h)! h!} \sum_{m_2,\dots,m_{h-1}\geqslant 0}  \prod_{k=2}^{h-1} \frac{1}{m_k!}\left(\frac{\Espe[W^k] \Espe[Z_r^k] h^k}{k! (r-h)^{k-1}}\right)^{m_k} \notag \\
 & = \frac{1}{(r-h)! h!} \exp\left\{\sum_{k=2}^{h-1} \frac{\Espe[W^k] \Espe[Z_r^k] h^k}{k! (r-h)^{k-1}} \right\} \label{eq:maj_Zrh_somme} .
% & = \frac{1}{(r-h)! h!} \exp\left\{\frac{h^2}{r-h} \left( \frac{\Espe[W^2]}{2} \Espe[Z_r^2]+ \sum_{l=1}^{h-3} \frac{\Espe[W^{l+2}]}{(l+2)!} \Espe[Z_r^{l+2}]  \left(\frac{h}{r-h} \right)^l\right) \right\} . 
\end{align}
Besides, using Stirling bounds \eqref{eq:stirling} and $\log(1+x) \leqslant x-x^2/2+x^3/3$ for all $x \geqslant 0$, we have
\begin{align}
\frac{r!}{(r-h)! r^h}
 & \leqslant \left(\frac{r}{r-h}\right)^{r-h+1/2} \exp\left\{-h+\frac{1}{12r}\right\} \notag \\
 & \leqslant \exp\left\{\left(r-h+\frac 12\right) \left(\frac {h}{r-h} -\frac 12 \left(\frac{h}{r-h}\right)^2+\frac 13 \left(\frac{h}{r-h}\right)^3\right) -h + \frac{1}{12r}\right\} \notag \\
 & \leqslant \exp\left\{\frac{h^2}{2(r-h)} \left(-1+\frac 1 h +\frac{2h}{3(r-h)} + \frac{h}{3(r-h)^2} + \frac{r-h}{6rh^2}\right)\right\} \notag \\
 & \leqslant \exp\left\{\frac{h^2}{2(r-h)} \left(-1 + \frac{13}{12 h} + \frac{h}{r-h} \right) \right\} . \label{eq:maj_Zrh_prefacteur}
\end{align}
Combining \eqref{eq:Zrh}, \eqref{eq:maj_Zrh_somme}, \eqref{eq:maj_Zrh_prefacteur}, and \eqref{eq:majoration_momemnt_W} together with $(1 - x)^{-1} \leqslant \exp(2x)$ for $x \in \intervalleff{0}{1/2}$, we get
\begin{align}
\Espe[Z_r^h]
 & \leqslant \exp\left\{\frac{h^2}{2(r-h)} \left(-1 + \frac{13}{12 h} + \frac{h}{r-h} \right) + \frac{2 \Espe[W^2]}{r} + \sum_{k=2}^{h-1} \frac{\Espe[W^k] \Espe[Z_r^k] h^k}{k! (r-h)^{k-1}} \right\} \notag \\
 & \leqslant \exp\left\{\frac{h^2}{2(r-h)} \left(-1 + \frac{C}{h} + \frac{h}{r-h} + 2 \sum_{k=2}^{h-1} \frac{\Espe[W^k] \Espe[Z_r^k]}{k!} \biggl( \frac{h}{r-h} \biggr)^{k-2} \right) \right\} \label{eq:maj_Zrh_total}
\end{align}
with $C = 2 \Espe[W^2] + 13/12$. By Proposition \ref{prop:moments_W}, item \ref{prop:moments_W_momk}, there exists $A > 0$ such that, for all $k \in \N$,
\[
\frac{\Espe[W^k]}{k!} \leqslant A^k .
\]
For all $\tilde{\eta} \in \intervalleoo{0}{1}$, let
\[
B(\tilde{\eta}) \defeq \exp \biggl\{ \frac{\tilde{\eta}}{2 (1 - \tilde{\eta})} \biggl( \Var(W) + \delta + \frac{C}{2} \biggr) \biggr\}.
\]
Assume now that $\eta$ is small enough that
\begin{align}
\frac{\eta}{1 - \eta} + \Espe[W^2] (B(\eta)^2 - 1) + 2 A^2 B(\eta)^2 \frac{A B(\eta) \frac{\eta}{1 - \eta}}{1 - A B(\eta) \frac{\eta}{1 - \eta}} < \delta . \label{eq:choix_eta}
\end{align}

% Let $r \geqslant 2 \vee (1/\eta)$. Denote by $E(h)$ the statement:
% \[
% \Espe[Z_r^h] \leqslant \exp\biggl\{ \frac{h^2}{2 (r - h)} \biggl( \Var(W) + \delta + \frac{C}{h} \biggr) \biggr\}
% \]
% and let us prove $E(h)$ by induction for $h \geqslant 2$. Since $\Espe[Z_r] = 1$, $E(1)$ is obvious. Now assume that $E(k)$ is satisfied for all $k \in \intervallentff{1}{h-1}$.

Using \eqref{eq:maj_Zrh_total} and noting that $h / (r - h) \leqslant \eta / (1 - \eta)$, one gets
\begin{align*}
\Espe[Z_r^h]
 & \leqslant \exp \biggl\{ \frac{h^2}{2 (r - h)} \biggl( -1 + \frac{C}{h} + \frac{\eta}{1 - \eta} + \Espe[W^2] B(\eta)^2 + 2 \sum_{k=3}^{h-1} A^k B(\eta)^k \Bigl( \frac{\eta}{1 - \eta} \Bigr)^{k-2} \biggr) \biggr\} \\
% & \leqslant \exp \biggl\{ \frac{h^2}{2 (r - h)} \biggl( \Var(W) + \frac{C}{h} + \frac{\eta}{1 - \eta} + \Espe[W^2] (B(\eta)^2 - 1) + 2 A^2 B(\eta)^2 \frac{A B(\eta) \frac{\eta}{1 - \eta}}{1 - A B(\eta) \frac{\eta}{1 - \eta}} \biggr) \biggr\} \\
 & \leqslant \exp\biggl\{ \frac{h^2}{2 (r - h)} \biggl( \Var(W) + \delta + \frac{C}{h} \biggr) \biggr\} ,
\end{align*}
bounding above the finite geometric sum by the value of the infinite sum, and using \eqref{eq:choix_eta}.
\end{proof}

The following estimate concerns the higher moments of $Z_r$.

\begin{prop}\label{prop:gros_moment}
For all $\eta \in \intervalleoo{0}{ce}$,
\[
\limsup_{r\to \infty} \frac 1r \log\Espe[Z_r^{\eta r}] <\infty .
\]
\end{prop}

% The following lemma will be useful in the proof of Proposition \ref{prop:gros_moment}.

% \begin{lem}\label{lem:convexity}
% Let $Z$ be a nonnegative random variable. 
% \begin{enumerate}
% \item The real-valued function $f\colon h \in \intervallefo{0}{\infty} \mapsto \log \Espe[Z^h]$ is convex. 
% \item If $\Espe[Z] = 1$, then the function $f$ is a nondecreasing function on $\intervallefo{1}{\infty}$.
% \end{enumerate}
% \end{lem}

% \begin{proof}[Proof of Lemma \ref{lem:convexity}]
% \pierre{06/10 -- C'est exactement le résultat de \cite[p.\ 132]{kahane1976} cité dans la Proposition \ref{prop:moments_W}. On peut donc remonter la proposition au début de la section, ou raccourcir la preuve.} First, let $\lambda \in \intervalleff{0}{1}$ and $x$ and $y$ real numbers. By Hölder's inequality, 
% \begin{align*}
% f(\lambda x+(1-\lambda)y)&\leqslant \log\left(\Espe[Z^x]^\lambda \Espe[Z^y]^{1-\lambda}\right)\\
% &=\lambda\log \Espe[Z^x]+(1-\lambda)\log \Espe[Z^y]\\
% &=\lambda f(x)+(1-\lambda) f(y).
% \end{align*}
% Second, if $\Espe[Z]=1$, then $f(0) = \log \Prob(Z > 0) \leqslant 0 = f(1)$. Hence the result.
% \end{proof}

\begin{proof}[Proof of Proposition \ref{prop:gros_moment}]
Let $r > \exp \Espe[W \log(W)]$. Let $\eta \in \intervalleoo{0}{ce}$. Let $\varepsilon \in \intervalleoo{0}{c}$ be such that $\eta < (c - \varepsilon) e$. As a consequence of Proposition \ref{prop:moments_W}, item \ref{prop:moments_W_momk}, there exists $D \geqslant 1$ such that, for all $h \geqslant 1$,
\begin{align} \label{eq:def_Rk}
\frac{\Espe[W^h]}{h!} \leqslant \frac{D}{(c - \varepsilon)^h} .
\end{align}
Let $\theta \in \intervalleoo{1/2}{1}$. For all $h \geqslant 2$, we write the Euclidean division $h = q \floor{\theta h} + s$. Using \eqref{eq:kahane_peyriere}, \eqref{eq:def_Rk}, and the convexity of $(h_1, \dots, h_r) \mapsto \log \prod_{j=1}^r \Espe[Z_r^{h_j}]$, we get
% Using \eqref{eq:kahane_peyriere} and Lemma \ref{lem:convexity} applied to $Z = Z_r$, one gets
% \begin{align} \label{eq:rec_2}
% \Espe[Z_r^h]
%  & \leqslant \frac{h!}{r^h \bigl( 1 - \frac{\Espe[W^h]}{r^{h-1}} \bigr)} \left( \frac{D^r}{(c - \varepsilon)^h} \sum_{0 \leqslant h_j < \theta h} \Espe[Z_r^{\theta h}]^{1/\theta} + \frac{D^{(1 - \theta) h}}{(c - \varepsilon)^h} \sum_{\exists h_j \geqslant \theta h} \Espe[Z_r^{h-1}] \right) \notag \\
%  & \leqslant \frac{h!}{r^h \bigl( 1 - \frac{\Espe[W^h]}{r^{h-1}} \bigr)}  \left( \frac{D^r}{(c - \varepsilon)^h} \binom{r+h}{r} \Espe[Z_r^{\theta h}]^{1/\theta} + \frac{D^{(1 - \theta) h}}{(c - \varepsilon)^h} r \binom{r + (1 - \theta) h}{r} \Espe[Z_r^{h-1}] \right)
% \end{align}
\begin{align*}
\Espe[Z_r^h]
 & \leqslant \frac{h!}{r^h \bigl( 1 - \frac{\Espe[W^h]}{r^{h-1}} \bigr)} \biggl( \frac{D^h}{(c - \varepsilon)^h} \sum_{0 \leqslant h_j \leqslant \theta h} \Espe[Z_r^{\floor{\theta h}}]^q \Espe[Z_r^s] + \frac{D^{(1 - \theta) h + 1}}{(c - \varepsilon)^h} \sum_{\exists h_j > \theta h} \Espe[Z_r^{h-1}] \Espe[Z_r] \biggr).
\end{align*}
The number of terms in the first sum is less than the total number of compositions of $h$ in $r$ parts, i.e.
\[
\binom{r + h - 1}{r - 1} \leqslant \binom{r + h}{r} ,
\]
and similarly the number of terms in the second sum is less than
\[
r \binom{r + h - \ceil{\theta h} - 2}{r - 2} \leqslant r \binom{r + (1 - \theta) h}{r} .
\]
(where, for all $x \in \R$, $\binom{x}{n} = x(x-1) \cdots (x-n+1) / n!$). Besides, noting
\[
\beta
 \defeq \frac{1}{\theta - \frac{1}{2}} + 1
 \geqslant \frac{1}{\theta - \frac{1}{h}} + 1
 \geqslant \floor{\frac{h}{\floor{\theta h}}} + 1
 = q + 1,
\]
and using the fact that $h \in \intervallefo{1}{\infty} \mapsto \Espe[Z_r^h]$ is nondecreasing and greater or equal to $1$,
\begin{align} \label{eq:rec_2}
\Espe[Z_r^h]
 & \leqslant \frac{h!}{r^h \bigl( 1 - \frac{\Espe[W^h]}{r^{h-1}} \bigr)}  \biggl( \frac{D^h}{(c - \varepsilon)^h} \binom{r+h}{r} \Espe[Z_r^{\theta h}]^\beta + \frac{D^{(1 - \theta) h + 1}}{(c - \varepsilon)^h} r \binom{r + (1 - \theta) h}{r} \Espe[Z_r^{h-1}] \biggr).
\end{align}

Our strategy is to bound the moments of $Z_r$ by backward induction, finally relying on the fact that the small moments are well bounded. Indeed, by Proposition \ref{prop:Zrh_rec}, there exists $\eta_1 \in \intervalleoo{0}{1 \wedge ce}$ such that
\[
\limsup_{r \to \infty} \frac{1}{r} \log \Espe[Z_r^{\eta_1 r}] < \infty .
\]
Let $\eta_0 = \eta_1 / 4$. We claim that, for all $\gamma \in \intervalleof{0}{1}$, there exists $C(\gamma) \in \R$ such that, for all $r$ large enough,
\begin{align} \label{eq:q_majo}
\forall h \in \intervallentff{\eta_0 r}{\eta r} \quad
\frac{h!}{r^h \bigl( 1 - \frac{\Espe[W^h]}{r^{h-1}} \bigr)} \binom{r + \gamma h}{r} \frac{D^{\gamma h}}{(c - \varepsilon)^h} \leqslant e^{C(\gamma) r} .
\end{align}
and $C(\gamma) < 0$ for $\gamma$ small enough. Indeed, using \eqref{eq:stirling} and applying Proposition \ref{prop:moments_W}, item \ref{prop:moments_W_terme_inverse},
\begin{align*}
\frac{1}{r} \log & \biggl( \frac{h!}{r^h \bigl( 1 - \frac{\Espe[W^h]}{r^{h-1}} \bigr)} \binom{r + \gamma h}{r} \frac{D^{\gamma h}}{(c - \varepsilon)^h} \biggr) \\
 & \leqslant \frac{h}{r} \log \left(\frac{h}{(c - \varepsilon) e r}\right) + \frac{h}{r} \gamma \log(D) + \left(1+\frac{\gamma h}{r} + \frac{1}{2r} \right)\log \left(1+\frac{\gamma h}{r} \right) - \frac{\gamma h}{r} \log \left( \frac{\gamma h}{r}\right) \\
 & \hspace{2cm} - \frac{1}{2 r} \log(\gamma) + \frac{2 \Espe[W^2]}{r^2} + \frac{1}{12 r h} + \frac{1}{12 r(r + \gamma h)} \\
 & \leqslant \eta_0 \log \left(\frac{\eta}{(c-\varepsilon) e}\right) + \gamma \eta \log(D) + (1 +\gamma \eta) \log(1 + \gamma \eta) + \min(e^{-1}, - \gamma \eta \log(\gamma \eta)) + \gamma ,
\end{align*}
for all $r$ large enough and the conclusion follows, noticing that the latter quantity is negative for small $\gamma > 0$.

From now on, we assume that $\theta \in \intervalleoo{1/2}{1}$ is such that $C(1 - \theta) < 0$. Let also $C_1 \defeq C(1)$. Hence, starting from \eqref{eq:rec_2} and using \eqref{eq:q_majo}, we get, for all $r$ large enough, for all $h \in \intervallentff{\eta_1 r}{\eta r}$,
\begin{align*}
\Espe[Z_r^h]
 & \leqslant e^{C_1 r} \Espe[Z_r^{\theta h}]^\beta + \frac{1}{2} \Espe[Z_r^{h-1}] \\
 & \leqslant e^{C_1 r} \Espe[Z_r^{\theta h}]^\beta + \frac{1}{2} \Bigl( e^{C_1 r} \Espe[Z_r^{\theta (h-1)}]^\beta + \frac{1}{2} \Espe[Z_r^{h-2}] \Bigr) \\ 
 & \leqslant e^{C_1 r} \Espe[Z_r^{\theta h}]^\beta \sum_{k=0}^{h - \floor{\theta h} - 1} \frac{1}{2^k} + \frac{1}{2^{h - \floor{\theta h}}} \Espe[Z_r^{\theta h}],
\end{align*}
by induction (note that the smallest moment to which we apply the induction argument is of order $\theta(\floor{\theta h} + 1) \geqslant \theta^2 h \geqslant h / 4 \geqslant \eta_0 r$, so \eqref{eq:q_majo} may be used) and the monotonicity of $h \mapsto \Espe[Z_r^{\theta h}]$. Since $\beta \geqslant 1$, $C_1 \geqslant 0$, and $\Espe[Z_r^{\theta h}] \geqslant 1$, we obtain
\begin{align}
\Espe[Z_r^h]
 & \leqslant e^{C_1 r} \Espe[Z_r^{\theta h}]^\beta \sum_{k=0}^{h - \floor{\theta h}} \frac{1}{2^k}
 \leqslant 2 e^{C_1 r} \Espe[Z_r^{\theta h}]^\beta . \label{eq:rec_theta_h}
\end{align}
Let $K \defeq \min\enstq{k \geqslant 1}{\theta^k \eta \leqslant \eta_1}$. Applying \eqref{eq:rec_theta_h} recursively, we get
\begin{align}
\log \Espe[Z_r^{\eta r}]
 & \leqslant \log(2) + C_1 r + \beta \log \Espe[Z_r^{\theta \eta r}] \notag \\
 & \leqslant (\log(2) + C_1 r)(1 + \beta) + \beta^2 \log \Espe[Z_r^{\theta^2 \eta r}] \notag \\
% & \leqslant \cdots \notag \\
 & \leqslant (\log(2) + C_1 r) \sum_{k=0}^{K-1} \beta^k + \beta^K \log \Espe[Z_r^{\theta^K \eta r}] . \label{eq:upper_bound_large_moment0}
\end{align}
Finally,
\[
\limsup_{r \to \infty} \frac{1}{r} \log \Espe[Z_r^{\eta r}]
 \leqslant C_1 \frac{\beta^K - 1}{\beta - 1} + \beta^K \limsup_{r \to \infty} \frac{1}{r} \log \Espe[Z_r^{\eta_1 r}]
 < \infty . \qedhere
\]
\end{proof}

Propositions \ref{prop:Zrh_rec} and \ref{prop:gros_moment} immediately entail the following result.

\begin{cor} \label{cor:moments} \leavevmode

\begin{enumerate}
\item \label{cor:moments_petits} For all $\alpha \in \intervallefo{0}{1/2}$ and $\zeta > 0$,
\[
\lim_{r \to \infty} \frac{1}{r^{2 \alpha - 1}} \log \Espe[Z_r^{\zeta r^\alpha}]
 = 1 .
\]

\item \label{cor:moments_moyens} For all $\alpha \in \intervallefo{1/2}{1}$ and $\zeta > 0$,
\[
\limsup_{r \to \infty} \frac{1}{r^{2 \alpha - 1}} \log \Espe[Z_r^{\zeta r^\alpha}]
 \leqslant \frac{\zeta^2 \Var(W)}{2} .
\]

\item \label{cor:moments_grands} For all $\eta \in \intervalleoo{0}{ce}$,
\[
\kappa(\eta)
 \defeq \limsup_{r \to \infty} \frac{1}{r} \log \Espe[Z_r^{\eta r}]
 < \infty
\]
and $\kappa(\eta) = O(\eta^2)$ as $\eta \to 0$.
\end{enumerate}
\end{cor}

\subsection{Proof of Theorem \ref{th:infinite}, item \ref{th:infiniteMD} (moderate deviations)}

\emph{Left deviations} --- Let $\alpha \in \intervalleoo{0}{1/2}$ and $a > 0$. We want to estimate $\Prob( r^\alpha (Z_r - 1) \leqslant - a)$. For all $t \leqslant 0$, we consider
\begin{align*}
\Lambda_r(s)
 & \defeq \frac{1}{r^{1-2\alpha}} \log \Espe[e^{s r^{1-\alpha}  (Z_r - 1)}] \\
 & = \frac{1}{r^{1-2\alpha}} \log \Espe\biggl[ \exp\biggl( r^{1-\alpha} s \frac{1}{r}\sum_{i=1}^r (W_i Z_{r,i} - 1) \biggr) \biggr] \\
 & = r^{2\alpha} \log \Espe[e^{s r^{-\alpha} (W Z_r - 1)}] .
\end{align*}
Using the Taylor expansion $e^x = 1 + x + x^2 e^{\theta(x)} / 2$ with $\theta(x) \leqslant \max(x, 0)$ and the fact that $\Espe[W Z_r] = \Espe[W] \Espe[Z_r] = 1$, we get
\begin{align*}
\Lambda_r(s)
 & = r^{2\alpha} \log\Bigl(1 + \frac{s^2}{2 r^{2 \alpha}} \Espe\bigl[ (W Z_r - 1)^2 e^{\theta(s r^{-\alpha} (W Z_r - 1))} \bigr] \Bigr)
 \xrightarrow[r \to \infty]{} \frac{s^2}{2} \Espe[(W - 1)^2] .
\end{align*}
Indeed,
\begin{align*}
\Bigl| \Espe\bigl[ & (W Z_r - 1)^2
  e^{\theta(s r^{-\alpha} (W Z_r - 1))} \bigr] - \Espe[(W - 1)^2] \Bigr| \\
 & \leqslant \abs{\Espe\bigl[ (W Z_r - 1)^2 \bigl( e^{\theta(s r^{-\alpha} (W Z_r - 1))} - 1 \bigr) \bigr]} + \abs{\Espe\bigl[ (W Z_r - 1)^2 - (W - 1)^2 \bigr]} \\
 & \leqslant \abs{\Espe\bigl[ (W Z_r - 1)^2 \bigr]} \bigl( e^{- s r^{-\alpha}} - 1 \bigr) + \abs{\Espe\bigl[ (W Z_r - 1)^2 - (W - 1)^2 \bigr]}  \xrightarrow[r \to \infty]{} 0
\end{align*}
by Corollary \ref{cor:moments}, item \ref{cor:moments_petits}. Finally, the unilateral version of Gärtner-Ellis theorem applies and gives
\[
\frac{1}{r^{1-2\alpha}} \log \Prob( r^\alpha (Z_r - 1) \leqslant - a)
 \xrightarrow[r \to \infty]{} - \frac{a^2}{2 \Var(W)} .
\]

\medskip

\emph{Right deviations} --- Let $a \geqslant 0$. Using \eqref{eq:E_infinite}, one gets
\[
\Prob(r^\alpha (Z_r - 1) \geqslant a)
 = \Prob\biggl( \sum_{i=1}^r (W_i Z_{r,i} - 1) \geqslant a r^{1-\alpha} \biggr) .
\]

Let us prove the lower bound. For all $\varepsilon > 0$,
\begin{align*}
\frac{1}{r^{1-2\alpha}} \log &\ \Prob\biggl( \sum_{i=1}^r (W_i Z_{r,i} - 1) \geqslant a r^{1-\alpha} \biggr) \\
 & \geqslant \frac{1}{r^{1-2\alpha}} \log \Prob\biggl( \sum_{i=1}^r (W_i Z_{r,i} - 1) \geqslant a r^{1-\alpha},\ \forall i \in \intervallentff{1}{r}\ Z_{r,i} \geqslant 1 - \varepsilon r^{-\alpha} \biggr) \\
 & \geqslant \frac{1}{r^{1-2\alpha}} \log \Prob\biggl( \sum_{i=1}^r (W_i (1 - \varepsilon r^{-\alpha}) - 1) \geqslant a r^{1-\alpha} \biggr) +  r^{2\alpha} \log \Prob(Z_r \geqslant 1 - \varepsilon r^{-\alpha}) .
\end{align*}
The left deviations for $Z_r$ yield
\[
\frac{1}{r^{1-2\alpha}} \log \Prob(Z_r - 1 < - \varepsilon r^{-\alpha})
  \xrightarrow[r \to \infty]{} - J(-\varepsilon)
 < 0 ,
\]
hence
\[
r^{2\alpha} \log \Prob(Z_r
 \geqslant 1 - \varepsilon r^{-\alpha})
  \xrightarrow[r \to \infty]{} 0 .
\]
Using \eqref{eq:modW} and letting $\varepsilon \to 0$, one finally gets
\begin{align*}
\liminf_{r \to \infty} \frac{1}{r^{1-2\alpha}} \log \Prob(r^\alpha (Z_r - 1) \geqslant a)
 \geqslant - J(a) .%\label{eq:moderate_infini_minoration}
\end{align*}
%%%%%%%%%%%%%%%%%%%%%%%%%%%%%%%%%%%%%%%%%%%%%%%%%%%%%%%%%%%%%%%%%%
%%%%%%%%%%%%%%%%%%%%%%%%%%%%%%%%%%%%%%%%%%%%%%%%%%%%%%%%%%%%%%%%%%
%%%%%%%%%%%%%%%%%%%%%%%%%%%%%%%%%%%%%%%%%%%%%%%%%%%%%%%%%%%%%%%%%%
Now, let us turn to the upper bound. Using Markov's inequality and Corollary \ref{cor:moments}, item \ref{cor:moments_moyens}, for any $\zeta > 0$,
\begin{align*}
\limsup_{r \to \infty} \frac{1}{r^{1-2\alpha}} \log \Prob\bigl(Z_r-1\geqslant a r^{-\alpha} \bigr) 
 & \leqslant \limsup_{r \to \infty} \biggl( \frac{1}{r^{1-2\alpha}}\log \Espe[Z_r^{\zeta r^{1-\alpha}}]-\zeta r^\alpha\log(1+a r^{-\alpha}) \biggr) \\
& \leqslant \frac{\zeta^2 \Var(W)}{2} - \zeta a .
\end{align*}
Otimizing in $\zeta$, we conclude to the required upper bound.

\subsection{Proof of Theorem \ref{th:infinite}, item \ref{th:infiniteLD} (large deviations)}

Remind that
\[
I^\infty = \lim_{n \to \infty} \downarrow I^n .
\]
In particular, for all $a \leqslant 1$, $I^\infty(a) = \Lambda_W^*(a)$ and the large deviations on the left for $(Z_r)_{r \geqslant 2}$ are given in Proposition \ref{prop:leftLD}. Let us turn to the large deviations on the right. 
Let us introduce, for all $a \geqslant 1$,
\begin{align}\label{def:Isup_inf}
-\underline{I}^{\infty}(a)
 \defeq \liminf_{r\to\infty} \frac 1r \log \Prob (Z_r\geqslant a)
\leqslant  \limsup_{r\to\infty} \frac 1r \log \Prob (Z_r\geqslant a)\eqdef -\overline{I}^{\infty}(a).
\end{align}

\begin{prop}[Lower bound] \label{prop:large_infini_lowerbound}
For all $a \geqslant 1$, $\underline{I}^\infty(a) \leqslant I^\infty(a)$.
\end{prop}

\begin{proof}[Proof of Proposition \ref{prop:large_infini_lowerbound}]
Remind that, for all $k \in \N^*$, $1^k$ represents the word $1, 1, \dots, 1$ of length $k$. For all $r \in \N^* \setminus \{ 1 \}$, $n \in \N^*$, $k \in \N$, and $i \in \intervallentff{1}{r}$, let
\begin{equation} \label{eq:Zr1kin}
Z_{r, 1^k, i}^n \defeq \frac{1}{r^n} \sum_{1 \leqslant i_1, \dots, i_n \leqslant r} W_{1^k, i, i_1} W_{1^k, i, i_1,i_2} \cdots W_{1^k, i, i_1,\dots, i_n}
\end{equation}
and let $Z_{r, 1^k, i}$ be the limit of the martingale $(Z_{r, 1^k, i}^n)_{n \geqslant 1}$ (by convention, $Z_{r, 1^0, i} = Z_{r, i}$). Note that $(W_{1^k, i})_{k \geqslant 0, 1 \leqslant i \leqslant r} \bullet (Z_{r, 1^k, i})_{k \geqslant 0, 2 \leqslant i \leqslant r}$ is a family of independent random variables. Now let $a \geqslant 1$, $n \in \N^*$, $\varepsilon > 0$, and consider $((s_1, w_1), (s_{1, 1}, w_{1, 1}), \dots, (s_{1^{n-2},1}, w_{1^{n-2},1}), z) \in (\R^2)^{n-1} \times \R$ such that
\begin{align} \label{eq:large_infini_lowerbound_contrainte}
s_1 (1-\varepsilon) + w_1\Bigl( 
s_{1,1}(1-\varepsilon) %+ w_{1,1}\Bigl( 
%s_{1,1,1}(1-\varepsilon)
 + \dots + w_{1^{n-3},1}
\Bigl(s_{1^{n-2},1}(1-\varepsilon) + w_{1^{n-2},1} z(1-\varepsilon) \Bigr) %\Bigr)
\Bigr)
 \geqslant a .
\end{align}
Then
\begin{align*}
\frac 1r \log \Prob (Z_r \geqslant a)
 \geqslant \frac 1r \log \Prob \Bigl(
 & W_1\geqslant w_1r,\ \sum_{i=2}^r W_i \geqslant s_1r,\ \forall i \in \intervallentff{2}{r}\ Z_{r, i} \geqslant 1 - \varepsilon, \\
 & W_{1,1}\geqslant w_{1,1} r, \sum_{i=2}^r W_{1,i} \geqslant s_{1,1}r,\ \forall i \in \intervallentff{2}{r}\ Z_{r, 1, i} \geqslant 1 - \varepsilon, \\
 & \cdots \\
 & W_{1^{n-2},1} \geqslant w_{1^{n-2},1} r,\ \sum_{i=2}^r W_{1^{n-2},i} \geqslant s_{1^{n-2},1}r,\ \forall i \in \intervallentff{2}{r}\ Z_{r, 1^{n-2}, i} \geqslant 1 - \varepsilon, \\
 & \sum_{i=1}^r W_{1^{n-1},i} \geqslant \frac{zr}{1-\varepsilon},\
\forall i \in \intervallentff{1}{r}\ Z_{r, 1^{n-1}, i} \geqslant 1-\varepsilon \Bigr) \\
 \geqslant \frac 1r \log \Prob \Bigl(
 & W_1\geqslant w_1r,\ \sum_{i=2}^r W_i \geqslant s_1r,\ W_{1,1}\geqslant w_{1,1} r,\ \sum_{i=2}^r W_{1,i} \geqslant s_{1,1}r,\ \cdots \\
 & W_{1^{n-2},1}\geqslant w_{1^{n-2},1} r,\ \sum_{i=2}^r W_{1^{n-2},i} \geqslant s_{1^{n-2},1}r,\ \sum_{i=1}^r W_{1^{n-1},i} \geqslant zr \Bigr) \\
 & + \frac{n (r - 1) + 1}{r} \log \Prob(Z_r \geqslant 1-\varepsilon).
\end{align*}
Remind that $\Prob(Z_r \geqslant 1-\varepsilon) \to 1$ as $r \to \infty$, so by independence,
\begin{align*}
-\underline{I}^{\infty}(a)
 & = \liminf_{r\to\infty} \frac 1r \log \Prob (Z_r\geqslant a)\\
 & \geqslant \liminf_{r\to\infty} \frac 1r \log \Prob \Bigl(
W_1\geqslant w_1r,\ \sum_{i=2}^r W_i \geqslant s_1r,\ W_{1,1}\geqslant w_{1,1} r,\ \sum_{i=2}^r W_{1,i} \geqslant s_{1,1}r, \\
 & \hspace{3.5cm} \cdots,\ W_{1^{n-2},1}\geqslant w_{1^{n-2},1} r,\ \sum_{i=2}^r W_{1^{n-2},i} \geqslant s_{1^{n-2},1}r, \\
 & \hspace{3.5cm} \sum_{i=1}^r W_{1^{n-1},i} \geqslant \frac{zr}{1-\varepsilon} \Bigr) \\
 & = -\bigl(cw_1+\Lambda_W^*(s_1)
+ cw_{1,1}+\Lambda_W^*(s_{1,1})
+ \dots + cw_{1^{n-2},1}+\Lambda_W^*(s_{1^{n-2},1})+\Lambda_W^*(z)
\bigr).
\end{align*}
Optimizing in $((s_1, w_1), (s_{1, 1}, w_{1, 1}), \dots, (s_{1^{n-2},1}, w_{1^{n-2},1}), z) \in (\R^2)^{n-1} \times \R$ satisfying \eqref{eq:large_infini_lowerbound_contrainte}
%i.e. 
%\begin{align*}
%s_1 + w_1\bigl( 
%s_{1,1} + w_{1,1}\bigl( 
%s_{1,1,1}
% + \dots +
%\bigl(s_{1^{n-2},1} + w_{1^{n-2},1} z \bigr) \bigr)
%\bigr)
% \geqslant \frac{a}{1 - \varepsilon} ,
%\end{align*}
and using the induction relation of Theorem \ref{th:finiteLD}, we get
\begin{align*}
-\underline{I}^{\infty}(a)
 \geqslant - I^n\Bigl( \frac{a}{1 - \varepsilon} \Bigr) .
\end{align*}
Taking the limit as $\varepsilon \to 0$ and using the continuity of $I^n$ (see Proposition \ref{prop:Indecroissante2}, item \ref{prop:Indecroissante_continuous}), and then taking the limit as $n \to \infty$, we conclude that
\begin{align*}
\underline{I}^{\infty}(a)
 & \leqslant \lim_{n \to \infty} I^n(a)
 = I^{\infty}(a) . \qedhere
\end{align*}
\end{proof}

Now, we turn to the upper bound, i.e.\ $\overline{I}^\infty \geqslant I^\infty$.

\begin{prop} \label{prop:pointfixe}
Let $a \geqslant 1$.
\begin{enumerate}
\item \label{prop:pointfixe_cinfini} If $c = \infty$, then
\[
\overline{I}^\infty(a) = \Lambda_W^*(a) = I^\infty(a) .
\]
\item \label{prop:pointfixe_cfini} If $c < \infty$,
\begin{align*}
\overline{I}^\infty(a)
 = \inf \ensavec{cw + \overline{I}^\infty(z)+\Lambda_W^*(s)}{w \geqslant 0,\ z \in \intervalleff{1}{a},\ s \in \intervalleff{1}{a},\ wz + s = a} .
% = \inf_{\substack{y\geqslant 0,\\ 1\leqslant z\leqslant a}}\Bigl[ \frac {cy}{z} + \overline{I}^\infty(z)+\Lambda_W^*(a-y)\Bigr]
% = \inf_{\substack{0\leqslant w\leqslant \frac{a-1}{z},\\ 1\leqslant z\leqslant a}}\Bigl[ cw + \overline{I}^\infty(z)+\Lambda_W^*(a-wz)\Bigr] .
\end{align*}
\end{enumerate}
\end{prop}

\begin{proof}[Proof of Proposition \ref{prop:pointfixe}]
Let $a \geqslant 1$. We follow the same lines as in Section \ref{sec:proof_finiteLD} for $Z_r^n$. Let $\varepsilon > 0$.  For all $r\in \N^* \setminus \{ 1 \}$ and for $q\in \intervallentff{0}{r}$, we define
\[
P_{r,q} \defeq \Prob\left(Z_r \geqslant a,\ Z_{r,1},\dots,Z_{r,q}\geqslant 1+\varepsilon,\ Z_{r,q+1},\dots, Z_{r,r} < 1+\varepsilon \right) ,
\]
so that 
\[
\Prob\left( Z_r \geqslant a \right)= \sum_{q=0}^r \binom{r}{q} P_{r,q} .
\]
Let $q_0 \in \intervallentff{0}{r}$ and $\varepsilon' \in \intervalleoo{0}{\varepsilon}$. By the definition of $\overline{I}^\infty$,
\begin{align*}
\sum_{q=q_0}^r \binom{r}{q} P_{r,q}
 &\leqslant \sum_{q=q_0}^r \binom{r}{q} \Prob(Z_r \geqslant 1+\varepsilon)^q \\
 &\leqslant \sum_{q=q_0}^r \bigl(re^{-r\overline{I}^{\infty}(1+\varepsilon')}\bigr)^q \\
 &\leqslant \frac{\bigl(re^{-r\overline{I}^{\infty}(1+\varepsilon')}\bigr)^{q_0}}{1-re^{-r\overline{I}^{\infty}(1+\varepsilon')}} .
\end{align*}

By Markov's inequality and Corollary \ref{cor:moments}, item \ref{cor:moments_grands}, for all $z > 1$, for all $\eta \in \intervalleoo{0}{ce}$,
\begin{align} \label{eq:Isupinfini_minoration}
- \overline{I}^\infty(z)
 = \limsup_{r \to \infty} \frac{1}{r}\log \Prob(Z_r\geqslant z)
 \leqslant \limsup_{r \to \infty} \frac 1r \log \Espe[Z_r^{\eta r}]-\eta \log(z)
 \leqslant \kappa(\eta)-\eta \log(z) . 
 %\sim -\eta \log(z)
\end{align}
%as $\eta \to 0$.
Hence $\overline{I}^\infty(z) > 0$, since $\kappa(\eta) = O(\eta^2)$ as $\eta \to 0$. Therefore, we may choose $q_0$ such that $q_0 \overline{I}^{\infty}(1+\varepsilon') > \overline{I}^{\infty}(a)$ and
\begin{align*}
\sum_{q=q_0}^r \binom{r}{q} P_{r,q}
 & \leqslant e^{-r \overline{I}^{\infty}(a)}
\end{align*}
as soon as $r$ is large enough. Now, 
\[
\lim_{r\to\infty} \frac 1r \log P_{r,0} \leqslant \lim_{r\to\infty} \frac 1r \log \Prob\Bigl(\sum_{i=1}^r W_i\geqslant \frac{ra}{1+\varepsilon}\Bigr)=  -\Lambda_W^* \Bigl(\frac{a}{1+\varepsilon}\Bigr).
\]
Finally, let $q\in \intervallentff{1}{q_0-1}$. Using the upper bound in the contraction principle for the continuous map $(y_1, \dots, y_q, s) \mapsto y_1 + \dots + y_q + (1+\varepsilon) s$ and $((W_1 Z_{r,1}, \dots, W_q Z_{r,q}, W_{q+1}+\dots+W_r) / r)_{r \geqslant 2}$, we get
\begin{align*}
\limsup_{r \to \infty} 
 &\ \frac 1r \log P_{r,q} \\
 & \leqslant \limsup_{r \to \infty} \frac 1r \log \Prob\left(W_1 Z_{r,1}+\dots+ W_q Z_{r,q}+\left(W_{q+1}+\dots+W_r\right)(1+\varepsilon) \geqslant ra, \right. \\
& \hspace{4cm} \left. Z_{r,1},\dots,Z_{r,q}\geqslant 1+\varepsilon,\ Z_{r,q+1},\dots,Z_{r,r}< 1+\varepsilon \right)\\
 & \leqslant \limsup_{r \to \infty} \frac 1r  \log \Prob\left(W_1 Z_{r,1}+\dots+ W_q Z_{r,q}+\left(W_{q+1}+\dots+W_r\right)(1+\varepsilon) \geqslant ra \right) \\
 & \leqslant  -\inf \Bigl\{g(y_1)+\dots+g(y_q)+\Lambda_W^*(s)\ ; \\
 & \hspace{3cm} y_1,\ \dots,\ y_q \geqslant 0,\ s \geqslant 0,\ y_1+\dots+y_q+(1+\varepsilon )s\geqslant a \Bigr\} \\
 & = -\inf \ensavec{g(y)+\Lambda_W^*(s)}{y \geqslant 0,\ s \geqslant 0,\ y + (1+\varepsilon) s \geqslant a} ,
\end{align*}
where $g$ is the concave function defined by
\begin{align*}
g(y)
 & \defeq \inf \ensavec{cw + \overline{I}^\infty(z)}{w \geqslant 0,\ z > 0,\ wz = y}  = \inf \ensavec{\frac {cy}{z} + \overline{I}^\infty(z)}{z > 0}
\end{align*}
and satisfying the upper bound of large deviations for $(W Z_r / r)_{r \geqslant 2}$ at speed $(r)_{r \geqslant 2}$. Then, proceeding as in the proof of Theorem \ref{th:finite}, item \ref{th:finiteLD}, we obtain
\begin{align*}
\overline{I}^\infty(a)
 & \geqslant \sup_{0 < \varepsilon \leqslant a-1} \inf \ensavec{cw + \overline{I}^\infty(z) + \Lambda_W^*(s)}{w \geqslant 0,\ z \in\intervalleff{1}{a},\ s \in \intervalleff{1}{\frac{a}{1+\varepsilon}},\ wz + (1 + \varepsilon)s = a} .
% & \geqslant \sup_{0 < \varepsilon < a-1} \inf_{\substack{y\geqslant 0,\\ z\geqslant 1}}\Bigl[ \frac {cy}{z} + \overline{I}^\infty(z)+\Lambda_W^*\Bigl(\frac{a-y}{1+\varepsilon}\Bigr)\Bigr] .
\end{align*}
If $c=\infty$, then, since $\Lambda_W^*$ is left continuous at $a$,
\begin{align*}
\overline{I}^\infty(a)
 \geqslant \sup_{\varepsilon >0} \Lambda_W^*\Bigl(\frac{a}{1+\varepsilon}\Bigr)
 = \Lambda_W^*(a)
 = I^\infty(a).
\end{align*}
Combining this inequality with \eqref{def:Isup_inf} and Proposition \ref{prop:pointfixe}, item \ref{prop:pointfixe_cinfini} is proved.

\medskip

Let us turn to item \ref{prop:pointfixe_cfini} and assume that $c<\infty$. We have
\begin{align*}
\overline{I}^\infty(a)
 & \geqslant\sup_{0<\varepsilon \leqslant a-1} \inf_{\substack{1 \leqslant z \leqslant a,\\ 1\leqslant s\leqslant \frac{a}{1+\varepsilon}}}\Bigl[ \frac {c(a-s)}{z} + \overline{I}^\infty(z)+\Lambda_W^*(s)-\frac{c\varepsilon s}{z}\Bigr]\\
 & \geqslant\sup_{0<\varepsilon \leqslant a-1} \inf_{\substack{1 \leqslant z \leqslant a,\\ 1\leqslant s\leqslant a}}\Bigl[ \frac {c(a-s)}{z} + \overline{I}^\infty(z)+\Lambda_W^*(s)-c\varepsilon a\Bigr]\\
  & = \inf_{\substack{1 \leqslant  \leqslant a \\ 1\leqslant s \leqslant a}}\Bigl[ \frac {c(a-s)}{z} + \overline{I}^\infty(z)+\Lambda_W^*(s)\Bigr]\\
  & = \inf \ensavec{c w + \overline{I}^\infty(z)+\Lambda_W^*(s)}{w \geqslant 0,\ z \in \intervalleff{1}{a},\ s \in \intervalleff{1}{a},\ wz + s = a} ,
\end{align*}
where the latter equality follows from the fact that $\overline{I}^\infty$ is nondecreasing on $\intervallefo{1}{\infty}$ by the very definition in \eqref{def:Isup_inf}, and even $\overline{I}^\infty(z) > \overline{I}^\infty(a)$ for all $z$ large enough by \eqref{eq:Isupinfini_minoration}, and $cw + \Lambda_W^*(s) > 0$ for all $(w, s) \in \intervallefo{0}{\infty} \times \intervallefo{1}{\infty} \setminus \{ (0, 1) \}$.

\medskip

Second, by definition and since $\Prob(Z_r\geqslant 1-\varepsilon)\to 1$ as $r\to \infty$, one has, for all $\varepsilon>0$, $z \geqslant 1$, and $s \in \intervalleff{1}{a}$,
\begin{align*}
-\overline{I}^\infty(a)
 & = \limsup_{r\to\infty} \frac 1r \log \Prob (Z_r\geqslant a) \\
 & \geqslant \limsup_{r\to\infty} \frac 1r \log \Prob (W_1Z_{r,1}+(W_2+\dots+W_r)(1-\varepsilon)\geqslant ra)+\frac{r-1}{r} \log \Prob(Z_r\geqslant 1-\varepsilon) \\
 & \geqslant \limsup_{r\to\infty} \frac 1r \log \Prob \Bigl(W_1\geqslant \frac{r(a - s(1 - \varepsilon))}{z},\ Z_{r,1}\geqslant z,\ \frac1r (W_2+\dots+W_r) \geqslant s \Bigr) \\
 & = - \Bigl( \frac{c(a - s)}{z}+\overline{I}^\infty(z)+\Lambda_W^*(s) + \frac{c s \varepsilon}{z} \Bigr)
\end{align*}
from which we deduce that
\begin{align*}
\overline{I}^\infty(a)
 & \leqslant \inf_{\varepsilon > 0} \inf_{\substack{1 \leqslant z \leqslant a \\ 1 \leqslant s \leqslant \frac{a}{1-\varepsilon}}} \Bigl[\frac{c(a - s)}{z}+\overline{I}^\infty(z)+\Lambda_W^*(s) + \frac{c s \varepsilon}{z} \Bigr] \\
%   & \leqslant \inf_{\substack{1 \leqslant z \leqslant a \\ 1 \leqslant s \leqslant a}} \inf_{\varepsilon > 0} \Bigl[\frac{c(a - s)}{z}+\overline{I}^\infty(z)+\Lambda_W^*(s) + \frac{c s \varepsilon}{z} \Bigr] \\
  & \leqslant\inf_{\substack{1 \leqslant z \leqslant a \\ 1 \leqslant s \leqslant a}} \Bigl[\frac{c(a - s)}{z}+\overline{I}^\infty(z)+\Lambda_W^*(s) \Bigr] \\
  & = \inf \ensavec{cw + \overline{I}^\infty(z)+\Lambda_W^*(s)}{w \geqslant 0,\ z \in \intervalleff{1}{a},\ s \in \intervalleff{1}{a},\ wz + s = a} . \qedhere
\end{align*}
%\[
%\overline{I}^\infty(a)
% = \inf_{\substack{y\geqslant 0,\\ 1\leqslant z\leqslant a}}\Bigl[ \frac {cy}{z} + \overline{I}^\infty(z)+\Lambda_W^*(a-y)\Bigr]
% = \inf_{\substack{1\leqslant z\leqslant a,\\ 0\leqslant w\leqslant \frac{a-1}{z}}} \Bigl[ cw + \overline{I}^\infty(z)+\Lambda_W^*(a-wz)\Bigr] . \qedhere
%\]
\end{proof}

In view of the latter proposition, it remains to prove the inequality $\overline{I}^\infty \geqslant I^\infty$ in the case $c < \infty$.

\begin{lem}\label{lem:Isup_continuous}
The nondecreasing function $\overline{I}^\infty$ is continuous.
\end{lem}

\begin{proof}[Proof of Lemma \ref{lem:Isup_continuous}]
Let $a_0 \geqslant 1$ and $\varepsilon > 0$. Since $\overline{I}^\infty$ is nondecreasing, it suffices to prove that there exists $a_1 > a_0$ such that $\overline{I}^\infty(a_1) \leqslant \overline{I}^\infty(a_0) + \varepsilon$. This follows the same lines as in the proof of Proposition \ref{prop:Indecroissante2}, item \ref{prop:Indecroissante_continuous}.
\end{proof}

For all $a \geqslant 1$, we may define
\[
z(a) \defeq \inf \argmin \ensavec{\overline{I}^\infty(z) + h(a, z)}{z \in \intervalleff{1}{a}}
\]
Recall that $a_W \defeq \inf \ensavec{a \geqslant 1}{a a^* \geqslant c}$.
Now if $a \in \intervalleff{1}{a_W}$ and $z \in \intervalleff{1}{a}$, then $z \leqslant c / a^*$, so $h(a, z) = \Lambda_W^*(a)$, whence $z(a) = 1$ and
\begin{align} \label{eq:Isupinfini_lambdastar}
\overline{I}^\infty(a)
 & = \Lambda_W^*(a)
 \geqslant I^\infty(a) .
\end{align}
\begin{lem}\label{lem:z_a}
For all $A > a_W$, there exists $\eta(A) > 0$ such that, for all $a \in \intervalleff{a_W}{A}$, $z(a) < a - \eta(A)$.
\end{lem}
\begin{proof}[Proof of Lemma \ref{lem:z_a}]
For all $a \geqslant a_W$ and $z \in \intervalleff{1}{a}$,
\begin{align*}
h(a, z)
 & \geqslant h(a, a)
 = c - \Lambda_W(c / a)
 \geqslant c - \Lambda_W(c / a_W)
 = \Lambda_W^*(a_W)
 \eqdef \rho
 > 0 .
\end{align*}
Now let $A > a_W$. Since the function $\overline{I}^\infty$ is uniformly continuous on $\intervalleff{1}{A}$, there exists $\eta(A) > 0$ such that, for all $(a, z) \in \intervalleff{1}{A}^2$, with $a - \eta(A) \leqslant z \leqslant a$, $\overline{I}^\infty(z) > \overline{I}^\infty(a) - \rho$ (note that it implies $\eta(A) < a_W - 1$). Then, for all $a \in \intervalleff{a_W}{A}$ and $z \in \intervalleff{a-\eta(A)}{a}$,
\begin{align*}
\overline{I}^\infty(z) + h(a, z)
 & > (\overline{I}^\infty(a) - \rho) + \rho
 = \overline{I}^\infty(a) ,
\end{align*}
whence the conclusion of the lemma.
\end{proof}
Let $E \defeq \ensavec{a \geqslant 1}{\forall z \in \intervalleff{1}{a}\ \overline{I}^\infty(z) \geqslant I^\infty(z)}$ and let $a_\infty = \sup(E)$. By \eqref{eq:Isupinfini_lambdastar}, we know that $a_\infty \geqslant a_W$. The aim is to prove that $a_\infty = \infty$. Assume that $a_\infty \in \intervallefo{a_W}{\infty}$. Let $A = a_\infty + 1 > a_W$ and $\eta(A)$ be defined as in Lemma \ref{lem:z_a}. Let $a$ be such that $a_\infty < a \leqslant a_\infty + \min(1, \eta(A))$. One has
\begin{align*}
\overline{I}^\infty(a)
 & = \inf_{1 \leqslant z < a - \eta(A)} \Bigl[ \overline{I}^\infty(z) + h(a, z) \Bigr] \\
 & = \inf_{1 \leqslant z < a - \eta(A)} \Bigl[ I^\infty(z) + h(a, z) \Bigr]
 \geqslant \inf_{1 \leqslant z \leqslant a} \Bigl[ I^\infty(z) + h(a, z) \Bigr]
 = I^\infty(a) ,
\end{align*}
so $a \in E$ which is a contradiction, and we conclude that $a_\infty = \infty$.

\subsection{Proof of Theorem \ref{th:infinite}, item \ref{th:infiniteVL} (very large deviations)}

\emph{Left deviations} ---
%By the law of large numbers, for $a>0$, $\Prob(Z_r\leqslant r^\alpha a) \to 1$ as $r \to \infty$ and,
For $a<0$, $\Prob(Z_r\leqslant r^\alpha a) = 0$. Assume now that $a=0$. Note that $\Prob(Z_r\leqslant 0) = \Prob(Z_r = 0)$. Let $p \defeq \Prob(W = 0)$. If $p = 0$, then $\Prob(Z_r = 0) = 0$. Otherwise, $0 < p < 1$ (recall that $\Espe[W]=1$ implies $p < 1$) and
\begin{align*}
\Prob(Z_r = 0)
% \geqslant \Prob(Z_r^n =0)
 \geqslant \Prob(Z_r^1 = 0)
 = p^r.
\end{align*}
Moreover, using \eqref{eq:E_infinite},
\begin{align*}
\Prob(Z_r=0)
 & = \Prob(WZ_r=0)^r
 = (1-\Prob(W\neq 0)\Prob( Z_r \neq 0))^r
 = (p+(1-p)\Prob(Z_r=0))^r .
\end{align*}
Then $\Prob(Z_r=0)$ is a fixed point of the strictly convex function $f \colon x \in \intervalleff{0}{1} \mapsto (p+(1-p)x)^r$.
For $r > \exp \Espe[W \log(W)]$, recalling that $\Prob(Z_r = 0) < 1$, $f$ has exactly two fixed points: $\Prob(Z_r = 0)$ and $1$.
%Remarque --- On peut aussi dire que $f'(1) = r(1-p) > 1$ et $f(0) = p^r > 0$. 
Let $\varepsilon>0$ and $x_{r, \varepsilon} \defeq p^{r(1-\varepsilon)}$. Then, for all $r$ large enough, $f(x_{r, \varepsilon}) < x_{r, \varepsilon}$, whence $\Prob(Z_r=0) < x_{r, \varepsilon}$.
% , leading to
% \[
% \frac{1}{r} \log \Prob(Z_r = 0) \leqslant (1-\varepsilon) \log(p).
% \]
Then we conclude that
\begin{align*}
-\frac{1}{r \log(r)} \log \Prob(Z_r\leqslant r^\alpha a) \to \begin{cases}
\infty & \text{for $a < 0$ or ($a = 0$ and $p = 0$)} \\
0 & \text{for ($a = 0$ and $p > 0$).}
%0 & \text{for $a>0$}\\
\end{cases}
\end{align*}

% \begin{align*}
% -\frac 1r \log \Prob(Z_r\leqslant r^\alpha a) \to \begin{cases}
% %0 & \text{for $a>0$}\\
% -\log(p) & \text{for $a=0$}\\
% \infty & \text{for $a<0$.}
% \end{cases}
% \end{align*}

\emph{Right deviations} ---  The case $a = 0$ is obvious. Assume $a > 0$. Let us prove the minoration. Recall the definition of $Z_{r, 1^n}$ given after Equation \eqref{eq:Zr1kin}. One has
\begin{align*}
\frac{1}{r\log(r)}\log \Prob(Z_r\geqslant r^\alpha a)
&\geqslant  \frac{1}{r\log(r)}\log \Prob\Bigl(\frac{W_{1^1}\cdots W_{1^n}}{r^n}\geqslant r^\alpha a,\, Z_{r, 1^n} \geqslant 1\Bigr)\\
&\geqslant \frac{1}{r\log(r)}\log \Prob\Bigl(\frac{W_{1^1}\cdots W_{1^n}}{r^n}\geqslant r^\alpha a\Bigr)+\frac{1}{r\log(r)}\log\Prob( Z_r \geqslant 1).
\end{align*}
By \cite[Theorem 1.2]{liu2000limit},
\[
\frac{1}{r\log(r)}\log\Prob( Z_r \geqslant 1)
 \xrightarrow[r \to \infty]{} 0.
\]
Now, $ n(r^\alpha a)^{1/n}$ is minimum for $n = n_r \defeq \floor{\log(r^\alpha a)}$. So
\begin{align*}
\frac{1}{r \log (r)} \log \Prob\Bigl(\frac{W_{1^1} \cdots W_{1^{n_r}}}{r^{n_r}}\geqslant r^\alpha a\Bigr)
&\geqslant \frac{n_r}{r \log (r)} \log \Prob\Bigl(W\geqslant r( r^\alpha a)^{1/n_r}\Bigr)
 \xrightarrow[r \to \infty]{} -c \alpha e.
%&\sim_{r\to \infty} -\log(r^\alpha a) c r( r^\alpha a)^{1/\log(r^\alpha a)}\\
%&= -c \alpha e r \log (r) - c e r \log(a)  \\
%&\sim_{r\to \infty}-c \alpha e .
\end{align*}

As for the upper bound, by Markov's inequality and Corollary \ref{cor:moments}, item \ref{cor:moments_grands}, for all $\eta \in \intervalleoo{0}{ce}$, there exists $\kappa(\eta) \in \intervalleoo{0}{\infty}$ such that, for all $r \geqslant 1$,
\begin{align*}
\log \Prob(Z_r\geqslant r^\alpha a)
 & \leqslant \log \Espe[Z_r^{\eta r}]-\eta r \log(r^\alpha a)\\
&\leqslant \kappa(\eta)r-\eta r \alpha \log(r) -\eta r \log(a)\\
& \sim -r\log(r)\eta \alpha ,
\end{align*}
as ${r\to \infty}$ and we conclude letting $\eta \to ce$.

\appendix

\section{Deviation estimates and large deviation principles}

\begin{lem} \label{lem:dev2ldp}
Let $(Y_r)_{r\geqslant1}$ be a sequence of real-valued random variables and $(v_r)_{r\geqslant 1}$ be a sequence of positive numbers diverging to infinity.
Assume that there exist $m \in \R$ and two functions $I_- \colon \intervalleof{-\infty}{m} \to \intervalleff{0}{\infty}$ and $I_+ \colon \intervallefo{m}{\infty} \to \intervalleff{0}{\infty}$ such that $I_-$ is decreasing (or infinite), $I_+$ is increasing (or infinite),
\[
\forall a \in \intervalleof{-\infty}{m} \quad \frac{1}{v_r} \log \Prob(Y_r \leqslant a) \xrightarrow[r \to \infty]{} - I_-(a)
\]
and
\[
\forall a \in \intervallefo{m}{\infty} \quad \frac{1}{v_r} \log \Prob(Y_r \geqslant a) \xrightarrow[r \to \infty]{} - I_+(a) .
\]
Then, the sequence $(Y_r)_{r \geqslant 1}$ satisfies a large deviation principle at speed $(v_r)_{r \geqslant 1}$ with rate function
\[
I(a) = \begin{cases}
I_-(a) & \text{if $a < m$} \\
\min\{I_-(m), I_+(m)\} & \text{if $a = m$} \\
I_+(a) & \text{if $a > m$.}
\end{cases}
\]
\end{lem}

\begin{proof}[Proof of Lemma \ref{lem:dev2ldp}]
Let $F$ be a closed subset of $\R$. Let us introduce
\[
a \defeq \sup(F \cap \intervalleof{-\infty}{m})
\quad \text{and} \quad
b \defeq \inf(F \cap \intervallefo{m}{\infty}) .
\]
By convention, $\sup(\emptyset) = -\infty$, $\inf(\emptyset) = \infty$, $I_-(-\infty) = I_+(\infty) = \infty$. Then, 
\begin{align*}
\limsup_{r \to \infty} \frac{1}{v_r} \log \Prob(Y_r \in F)
 & \leqslant \limsup_{r \to \infty} \frac{1}{v_r} \log \Prob(Y_r \in \intervalleof{-\infty}{a} \cup \intervallefo{b}{\infty})\\
 & \leqslant -\min \{I_-(a),I_+(b)\}\\
 & = -\inf_{x \in F} I(x),
\end{align*}
using the fact that $I_-$ is decreasing over $\intervalleof{-\infty}{m}$ and $I_+$ is increasing over $\intervallefo{m}{\infty}$. Now, let $G$ be an open subset of $\R$. Let $x \in \R$ and $\varepsilon > 0$ be such that $\intervalleoo{x-\varepsilon}{x+\varepsilon} \subset G$. Assume that $x \geqslant m$ (the case $x \leqslant m$ is treated analogously). Now,
\begin{align*}
\liminf_{r \to \infty} \frac{1}{v_r} \log \Prob(Y_r \in G)
 & \geqslant \liminf_{r \to \infty} \frac{1}{v_r} \log \Prob(Y_r \in \intervalleoo{x-\varepsilon}{x+\varepsilon}) \\
 & \geqslant \liminf_{r \to \infty} \frac{1}{v_r} \log \Prob(Y_r \in \intervallefo{x}{x+\varepsilon})\\
 & \geqslant \liminf_{r \to \infty} \frac{1}{v_r} \log \bigl(\Prob(Y_r \geqslant x)- \Prob(Y_r \geqslant x+\varepsilon)\bigr)\\
 & = -I_+(x)
\end{align*}
since
\[
\lim_{r \to \infty} \frac{1}{v_r} \log \Prob(Y_r \geqslant x)
 = -I_+(x)
 > -I_+(x+\varepsilon)
 = \lim_{r \to \infty} \frac{1}{v_r} \log \Prob(Y_r \geqslant x+\varepsilon) . \qedhere
\]
\end{proof}

% The following result is a unilateral version of the contraction principle for the sum of two nonnegative sequences. It is a particular case of \cite[Proposition 6]{FATP2020}.

% \begin{prop}[Unilateral sum-contraction principle] \label{principe-contraction-couple} Let  $(Y_{n,1} ,Y_{n,2})_{n\geqslant 0}$ be a sequence of $(\R_+)^2$-valued random variables such that, for each $n$, $Y_{n,1}$ and $Y_{n,2}$ are independent. Let $(v_n)_{n\geqslant 0}$ be a positive sequence diverging to $\infty$. For all $a \in \R$ and $i \in \{ 1, 2 \}$, assume that
% \[
% \lim_{n\to \infty} \frac{1}{v_n} \log\Prob(Y_{n,i} \geqslant a) \eqdef - I_i(a).
% \]
% Then, for all $c\in\R$, one has
% \[
% \lim_{n\to \infty} \frac{1}{v_n} \log\Prob(Y_{n,1}+Y_{n,2}\geqslant c) = - \inf_{\substack{ a,b \in \R \\ a+b=c}} (I_1(a) + I_2(b)).
% \]
% \end{prop}

\textbf{Acknowledgement}  We gratefully thank Alain Rouault for making us discover this problem and for the fruitful discussions that followed.

%\newpage
\bibliographystyle{abbrv}
\bibliography{biblio_gde_dev_alpha}

\end{document}